\documentclass[11pt,reqno]{amsart}
\usepackage{amssymb,amsfonts,amsthm}
\usepackage{amsthm}
\usepackage{amssymb}
\usepackage[francais,english]{babel}

\usepackage{color}

\usepackage{amsfonts,amsmath,layout}

\newcommand{\N}{\mathbb N}

\newcommand{\Z}{\mathbb Z}

\newcommand{\R}{\mathbb R}

\newcommand{\T}{\mathbb{T}}

\def\R{\mathbb R}
\def\N{\mathbb N}
\def\Z{\mathbb Z}
\def\E{\mathbb E}

\def\ep{\epsilon}
\def\rg{\rangle} 
\def\lg{\langle}

\newcommand{\be}{\begin{equation}}
\newcommand{\ee}{\end{equation}}

\def\1{{\bf 1}}
\def\rife#1{(\ref{#1})}

\def\inte{\int_{\T^d}}

\def\de{\delta}

\def\vep{\varepsilon}
\def\pelle#1{L^{#1}((0,T)\times \T^d)}
\def\elle#1{L^{#1}(\T^d)}
\def\intq{\int_0^T\int_{\T^d}}
\def\limitate#1{L^\infty((0,T); \elle{#1})}

\def\ds{\displaystyle}

\def\eps{\epsilon}

\newcommand{\s}[1]{{\mathcal #1}}

\newcommand{\bb}[1]{{\mathbb #1}}

\newtheorem{Theorem}{Theorem}[section]
\newtheorem{Definition}[Theorem]{Definition}
\newtheorem{Proposition}[Theorem]{Proposition}
\newtheorem{Lemma}[Theorem]{Lemma}

\newtheorem{Remark}[Theorem]{Remark}

\begin{document}

\title{Second order mean field games with degenerate diffusion and local coupling}
\author{Pierre Cardaliaguet}
\address{Ceremade, Universit\'e Paris-Dauphine,
Place du Mar\'echal de Lattre de Tassigny, 75775 Paris cedex 16 - France}
\email{cardaliaguet@ceremade.dauphine.fr }
\author{P. Jameson Graber}
\address{828, Boulevard des Mar\'echaux, 91762 Palaiseau Cedex}
\email{jameson.graber@ensta-paristech.fr}
\author{Alessio Porretta}
\address{Dipartimento di Matematica, Universit\`a di Roma ``Tor Vergata'',
Via della Ricerca Scientifica 1, 00133 Roma (Italy)}
\email{porretta@mat.uniroma2.it}
\author{Daniela Tonon}
\address{Ceremade, Universit\'e Paris-Dauphine,
Place du Mar\'echal de Lattre de Tassigny, 75775 Paris cedex 16 - France}
\email{tonon@ceremade.dauphine.fr}
\email{ }

\dedicatory{Version: \today}

\maketitle

\begin{abstract} 
We analyze a (possibly degenerate)    second order mean field games system of partial differential equations.
The distinguishing features of the model considered are (1) that it is not uniformly parabolic,  including the first order case as a possibility, and (2) the coupling is a local operator on the density. As a result we look for weak, not smooth, solutions.
Our main result is the existence and uniqueness of suitably defined weak solutions, which are characterized as minimizers of two optimal control problems.
We also show that such solutions are stable with respect to the data, so that in particular the degenerate case can be approximated by a uniformly parabolic (viscous) perturbation.
\end{abstract}

\section*{Introduction}

This paper is devoted to the analysis of  second order  mean field games systems with a local coupling. The general form of these systems is: 
\be\label{MFG}
\left\{\begin{array}{cl}
(i)&- \partial_t \phi -A_{ij}\partial_{ij}\phi+H(x,D\phi) =f(x,m(x,t))\\
(ii) & \partial_t m-\partial_{ij}( A_{ij}m)-{\rm div} (mD_pH(x,D\phi))=0\\
(iii)& m(0)=m_0, \; \phi(x,T)=\phi_T(x)
\end{array}\right.
\ee
where $A:\R^d\to \R^{d\times d}$ is symmetric and nonnegative, the Hamiltonian $H:\R^d\times \R^d\to\R$ is convex in the second variable, the coupling
$f:\R^d\times [0,+\infty)\to [0,+\infty)$ is increasing with respect to the second variable, $m_0$ is a probability density and $\phi_T:\R^d\to\R$ is a given function.  The functions $H$ and $f$, and  the matrix $A$, could as well depend on time, but since this does not give any  additional difficulty, we will avoid it just to simplify notations.

Mean field game systems (MFG systems) have been introduced simultaneously by Lasry-Lions \cite{LL06cr1, LL06cr2, LL07mf, LLperso} and Huang-Caines-Malham\'e \cite{HCMieeeAC06} to describe Nash equilibria in differential games with infinitely many players. The first unknown $\phi=\phi(t,x)$ is the value function of an optimal control problem of a typical small player. In this control problem, the dynamics is given by the controlled stochastic differential equation
$$
dX_s= v_s ds + \Sigma(X_s)dB_s,
$$
where $(v_s)$ is the control, $(B_s)$ is a Brownian motion and $\Sigma\Sigma^T=A$. The cost is given by 
$$
\E\left[ \int_0^T H^*(X_s, -v_s)+ f(X_s, m(s,X_s))\ ds+ \phi_T(X_T)\right]
$$
For each time $t\in [0,T]$ the quantity $m(t,x)$ denotes the density of population of small players at position $x$. In the control problem the term involving $f$ formalizes the fact that the cost of the player depends on this density $m$. As $\phi$ is the value function of this control problem, the optimal control of a typical small player is formally given by the feedback $(t,x)\to -D_pH(x,D\phi(t,x))$. Hence 
the second equation \eqref{MFG}-(ii) is the Kolmogorov equation of the process $(X_s)$ when the small player  plays in an optimal way. By the  mean field approach, this equation also describes the evolution of the whole population density as all players play in an optimal way. 

MFG systems with uniformly parabolic diffusions---typically $A_{ij}\partial_{ij}\phi=\Delta \phi$---have been the object of several contributions, either by PDE methods (see, e.g., \cite{CLLP, LL06cr1, LL06cr2, LL07mf, LLperso, GPSM1, GPSM2, Po}) or by stochastic techniques (see, e.g., \cite{CaDe2, HCMieeeAC06}): in this setting one often expects the solutions to be smooth, at least if the coupling is nonlocal and regularizing or if it has a ``small growth". The case of local couplings with an arbitrary growth has been discussed in \cite{CLLP} for purely quadratic hamiltonians (i.e. $H=|D\phi|^2$), in which case solutions are proved to be smooth,   and in \cite{Po} for general hamiltonians, by proving existence and uniqueness of weak solutions. 

Here we concentrate on degenerate parabolic equations. In this case the usual fixed point techniques used to prove the existence of solutions in the uniformly parabolic setting break down by lack of regularity. One then has to rely on convex optimization methods: this idea, which goes back to the analysis of some optimal transport problems (see \cite{bb,CCN}), has already been used to study first order MFG systems (i.e., $A\equiv 0$): see \cite{c1, cg, Gra14}. However it was not clear in these papers wether the weak solution was stable with respect to viscous approximation, i.e., if we could obtain weak solutions of the first order MFG systems by passing to the limit in uniformly parabolic ones.  This issue has partially motivated our study. 

In this paper we show the existence and uniqueness of  a weak solution for the degenerate mean field game system \eqref{MFG} as well as  the stability of solutions with respect to perturbation of the data: this includes of course stability by viscous approximation.
 
Concerning existence and uniqueness of solutions, the paper improves the existing results in two directions. First we consider non uniformly parabolic second order MFG systems, which have never been considered before. The introduction of second order derivatives induces several issues: in particular, in contrast with the first order equations, we do not expect the function $\phi$ to be BV (as in \cite{cg,Gra14}), which obliges us to be very careful about trace properties. Secondly---and this is new even for first order MFG systems---we drop a restriction between the growth condition of $H$ and the growth condition of $f$, restriction which was mandatory in the previous papers: see \cite{c1, cg}. To overcome the difficulty, we provide new  {\it integral estimates} for  subsolutions of Hamilton-Jacobi equations with unbounded right-hand side (Theorems \ref{int-reg} and  \ref{thm:estiHJ}). We think that these results are of independent interest. 

With these estimates in hand, the structure of proof for the existence and uniqueness follows roughly the lines already developed in \cite{c1, cg, CCN, Gra14}: basically it amounts to show that the MFG system can be viewed as an optimality condition for two convex problems, the first one being an optimal control of Hamilton-Jacobi equation, the second one an optimal control problem for the Fokker-Planck equation (see section \ref{sec:2opti} for details). A byproduct of this approach is  the stability of weak solutions with respect to the data (Theorem \ref{thm:stab}), which can be obtained by $\Gamma-$convergence techniques.

The paper is organized as follows. First  we introduce the notation and assumptions  needed throughout the paper (section \ref{sec:notass}). Then (section \ref{sec:esti})  we give our new estimates for subsolutions of Hamilton-Jacobi equations with a superlinear growth in the gradient variable and an unbounded right-hand side. In section \ref{sec:2opti}, we introduce the two optimal control problems and show that they are in duality while in section \ref{sec:optiHJ} we show that the optimal control problem for the Hamilton-Jacobi equation has a ``relaxed solution." Section \ref{sec:exuniq} is devoted to the analysis of the MFG system (existence, uniqueness and characterization). In the last section we discuss the stability of solutions. 

\bigskip

{\bf Acknowledgement: } 
This work has been partially supported by the Commission of the
European Communities under the 7-th Framework Programme Marie
Curie Initial Training Networks   Project SADCO,
FP7-PEOPLE-2010-ITN, No 264735,  by the French National Research Agency
 ANR-10-BLAN 0112 and ANR-12-BS01-0008-01 and by the Italian Indam Gnampa project 2013 \lq\lq Modelli di campo medio nelle dinamiche di popolazioni e giochi differenziali\rq\rq.

\section{Notations and assumptions}\label{sec:notass}

\noindent {\bf Notations :} We denote by 
 $\lg x, y\rg$ the Euclidean scalar product  of two vectors $x,y\in\R^d$ and by
$|x|$ the Euclidean norm of $x$. We use conventions on repeated indices: for instance, if $a,b\in \R^d$, we often write $a_ib_i$ for the scalar product $\lg a,b\rg$. More generally, if $A$ and $B$ are two square symmetric matrices of size $d\times d$, we write $A_{ij}B_{ij}$ for ${\rm Tr}(AB)$.  


To avoid further difficulties arising from boundary issues, we work in the flat $d-$dimen\-sional torus $\T^d=\R^d\backslash \Z^d$. We denote by $P(\T^d)$ the set of Borel probability measures over $\T^d$. It is endowed with the weak convergence. For $k,n\in\N$ and $T>0$, we denote by ${\mathcal C}^k([0,T]\times \T^d, \R^n)$ the space of maps $\phi=\phi(t,x)$ of class ${\mathcal C}^k$ in time and space with values in $\R^n$. For $p\in [1,\infty]$ and $T>0$, we denote by $L^p(\T^d)$ and $L^p((0,T)\times \T^d)$ the set of $p-$integrable maps over $\T^d$ and $[0,T]\times \T^d$ respectively. We often abbreviate $L^p(\T^d)$ and $L^p((0,T)\times \T^d)$ into  $L^p$. We denote by $\|f\|_p$ the $L^p-$norm of a map $f\in L^p$. \\

\noindent {\bf Assumptions:} We now collect the assumptions on the coupling $f$, the Hamiltonian $H$ and the initial and terminal conditions $m_0$ and $\phi_T$. These conditions are supposed to hold throughout the paper. 
\begin{itemize}
\item[(H1)] (Condition on the coupling) the coupling $f:\T^d\times [0,+\infty)\to \R$ is continuous in both variables, increasing with respect to the second variable $m$, 
and there exist $q>1$  and $C_1$  such that 
\be\label{Hypf}
 \frac{1}{C_1}|m|^{q-1}-C_1\leq f(x,m) \leq C_1 |m|^{q-1}+C_1 \qquad \forall m\geq 0 \;.
\ee
Moreover we ask the following normalization condition to hold:
\be\label{Hypf(0,m)=0}
f(x,0)=0 \qquad \forall x\in \T^d\;.
\ee
We denote by  $p$  the conjugate of $q$: $1/p+1/q=1$. 
\item[(H2)] (Conditions on the Hamiltonian) The Hamiltonian  $H:\T^d\times \R^d\to\R$ is continuous in both variables, convex and differentiable in the second variable, with $D_pH$ continuous in both variables, and has a superlinear growth in the gradient variable: there exist $r>1$ and  $C_2>0$ such that
\be\label{HypGrowthH}
\frac{1}{rC_2} |\xi|^{r} -C_2\leq H(x,\xi) \leq \frac{C_2}{r}|\xi|^r+C_2\qquad \forall (x,\xi)\in \T^d\times \R^d\;.
\ee
We note for later use that the Fenchel conjugate $H^*$ of $H$ with respect to the second variable is continuous and satisfies similar inequalities
\be\label{HypHstar}
\frac{1}{r'C_2} |\xi|^{r'} -C_2\leq H^*(x,\xi) \leq \frac{C_2}{r'}|\xi|^{r'}+C_2\qquad \forall (x,\xi)\in \T^d\times \R^d\;,
\ee
where $r'$ is the conjugate of $r$: $\ds\frac{1}{r}+\frac{1}{r'}=1$. 

\item[(H3)] (Conditions on $A$) there exists a Lipschitz continuous map $\Sigma:\T^d\to \R^{d\times D}$  such that $\Sigma\Sigma^T=A$ : let $C_3$ be a constant such that
\be\label{HypA}
|\Sigma(x)-\Sigma(y)|\leq C_3 |x-y|  \qquad \forall x,y\in \T^d,
\ee
Moreover we suppose that
\be\label{CondCroiss}
{\rm either }\; r\geq  p\qquad {\rm or}\qquad A\equiv 0.
\ee
We recall that  $p$ is the conjugate of $q$.

\item[(H4)] (Conditions on the initial and terminal conditions) $\phi_T:\T^d\to \R$ is of class ${\mathcal C}^2$, while $m_0:\T^d\to \R$ is a $C^1$ positive density (namely $m_0>0$ and $\ds \int_{\T^d} m_0dx=1$). 
\end{itemize}

Condition \eqref{Hypf(0,m)=0} is just a normalization condition, which we may assume  without loss of generality.  
Indeed,  if all the conditions (H1)$\dots$(H4) but \eqref{Hypf(0,m)=0} hold, then one just needs to replace $f(x,m)$ by $f(x,m)-f(x,0)$ and $H(x,p)$ by $H(x,p)-f(x,0)$: the new $H$ and $f$ still satisfy the above conditions (H1)$\dots$(H4) with
\eqref{Hypf(0,m)=0}. 

Let us set 
$$\ds F(x,m)=
\left\{\begin{array}{ll}
\ds \int_0^m f(x,\tau)d\tau & {\rm if }\; m\geq 0\\
+\infty & {\rm otherwise}
\end{array}\right. 
$$
Then $F$ is   continuous on $\T^d\times (0,+\infty)$, differentiable and strictly convex in $m$  and satisfies 
\be\label{HypGrowthF}
 \frac{1}{qC_1}|m|^{q}-C_1\leq F(x,m) \leq  \frac{C_1}{q}|m|^{q}+C_1 \qquad \forall m\geq 0
\ee
(changing the constant $C_1$ if necessary). 
Let $F^*$ be the Fenchel conjugate of $F$ with respect to the second variable. Note that $F^*(x,a)=0$ for $a\leq 0$ because $F(x,m)$ is nonnegative and equal to $+\infty$ for $m<0$.  Moreover, 
\be\label{HypGrowthFstar}
\frac{1}{pC_1}|a|^{p}-C_1\leq F^*(x,a) \leq  \frac{C_1}{p}|a|^{p}+C_1 \qquad \forall a\geq 0\;.
\ee

\section{Basic estimates on solutions of Hamilton-Jacobi equations}\label{sec:estimates}\label{sec:esti}

In this section we prove estimates in Lebesgue spaces for subsolutions of Hamilton-Jacobi equations of the form 
\be\label{eq:HJ}
\left\{\begin{array}{cl}
(i)&- \partial_t \phi -A_{ij}(x)\partial_{ij}\phi+H(x,D\phi) \leq \alpha(t,x)\\
(ii)&  \phi(x,T)\leq\phi_T(x)
\end{array}\right.
\ee
in terms of Lebesgue norms of $\alpha$ and  $\phi_T$. We assume that \eqref{HypGrowthH} and \eqref{HypA} hold, and \eqref{eq:HJ} is understood in the sense of distributions. This means that $D\phi\in L^r$ and, for any nonnegative test function $\zeta\in C^\infty_c((0,T]\times \T^d)$, 
$$
-\int_{\T^d}\zeta(T)\phi_T+ \int_{0}^{T}\int_{\T^d} \phi\partial_t\zeta+ \lg D\zeta, AD\phi\rg + \zeta(\partial_iA_{ij}\partial_j\phi+ H(x,D\phi)) \leq \int_{0}^{T}\int_{\T^d} \alpha \zeta.
$$
The estimates will be a consequence of the divergence structure of second  order terms. 

\begin{Theorem}\label{int-reg}
Assume that   $\phi\in L^r((0,T);W^{1,r}(\T^d))$ is a nonnegative function satisfying, in distributional sense,
\be\label{ineq}
\left\{\begin{array}{cl}
(i)&- \partial_t \phi - \partial_i \left( A_{ij}(x)\partial_{j}\phi\right)+ c_0\, |D\phi|^r \leq \alpha(t,x)\\
(ii)&  \phi(x,T)\leq\phi_T(x)
\end{array}\right.
\ee
for some nonnegative, bounded Lipschitz  matrix $A_{ij}$, and some $r>1$, $c_0>0$, $\alpha \in \pelle p$ and $\phi_T\in \elle \infty$.
Then, there exists a constant $C=C(p,d,r,c_0, T, \|\alpha\|_{L^p((0,T)\times \T^d)}, \|\phi_T\|_{\elle \eta})$ such that 
$$
\|\phi\|_{L^\infty((0,T), L^{\eta}(\T^d))}+ \|\phi\|_{L^\gamma((0,T)\times \T^d)}\leq C 
$$
where $\eta=\frac{d( r(p-1)+1)}{d-r(p-1)}$ and $\gamma = \frac{rp(1+d)}{d-r(p-1)}$ if $p<1+\frac dr$ and $\eta=\gamma=+\infty$ if $p>1+\frac dr$. 
\end{Theorem}

We note for later use that $\gamma>r$. 

\begin{proof}
Up to a rescaling, we may assume that $c_0=1$. We first claim that, for any real function $g\in W^{1,\infty}(\R)$ which is nondecreasing, and nonnegative in $\R_+$, we have
\be\label{chain}
 \int_{\T^d}  G(\phi(\tau))  \, dx   +   \int_\tau^T \int_{\T^d} |D\phi  |^r\, g(\phi)\,dxdt \leq  \int_\tau^T \int_{\T^d} \alpha\, g(\phi)\, dxdt+ \int_{\T^d}   G(\phi_T)\,  dx
\ee
for a.e. $\tau\in (0,T)$, where $G(r)=\int_0^r g(s)\,ds$.

There are several possible ways to justify \eqref{chain}, one is to use  regularization.

We first extend $\phi$  to $(0,T+1]\times \T^d$ by defining  $\phi= \phi_T$ on $[T,T+1]$. Then it still holds in the sense of distributions 
\be\label{extineq}
- \partial_t \phi - \partial_i \left( \tilde A_{ij}(t,x)\partial_{j}\phi\right)+  |D\phi|^r\,\chi_{(0,T)} \leq \tilde \alpha(t,x)
\ee
where $\tilde A_{ij}(t,x)=  A_{ij}(x)\chi_{(0,T)}(t)$ and $\tilde \alpha(t,x)= \alpha (t,x)\chi_{(0,T)}(t)$.
Let $\xi$ be a standard convolution kernel in $(t,x)$ defined  on $\R^{d+1}$ and $\xi^\eps(t,x)=\xi((t,x)/\eps)/(\eps)^{d+1}$, $\xi^\eps\geq0,\ \int_{\R^{d+1}}\xi^\eps(t,x) dtdx=1$, for all $\eps>0$. Let $\phi_\eps= \xi^\eps\star \phi$ and $\alpha_\eps=\xi^\eps\star \tilde \alpha$. Then  $\phi_\eps, \alpha_\eps$ are $C^\infty$ and converge  to $\phi,\alpha$ in their respective Lebesgue spaces.
 Convolving $\xi^\eps$ with (\ref{extineq}) we obtain on $(0,T+1)\times \T^d$: 
$$
- \partial_t \phi_\eps - \partial_i \left( \tilde A_{ij}(t,x)\partial_{j}\phi_\eps\right)+    |\xi_\eps\star(D\tilde \phi)|^r \leq \alpha_\eps(t,x)+R_\eps
$$
 where $D\tilde \phi= D\phi\chi_{(0,T)}$  and we used the fact that $D\phi\mapsto |D\phi|^r$ is convex,  and where  $R_\eps= -\partial_i \left( \tilde A_{ij}(t,x)\partial_{j}\phi_\eps\right)+\xi_\eps\star \partial_i \left(  \tilde A_{ij}(t,x)\partial_{j}\phi\right)$.
 
Using the notation (cf. \cite{DL89})
$$
[\xi^\epsilon,c](f) := \xi^\epsilon \star (cf) - c(\xi^\epsilon \star f)
$$
we can rewrite $R_\eps$ as $R_\eps= [\xi_\eps, \partial_i  \tilde A_{ij}](\partial_j \phi)+[\xi_\eps,  \tilde A_{ij}\partial _i](\partial_j \phi)$.
Invoking \cite[Lemma II.1]{DL89}, we have that $R_\epsilon \to 0$ in $L^r$, since $D\phi \in L^r$ and $ A_{ij} $ is Lipschitz.

Multiplying by $g(\phi_\eps)$ and integrating over $[\tau,T+\vep]\times \T^d$, for  $\tau\in (0,T)$, it follows 
\begin{align*} &\int_{\T^d}G(\phi_\eps(\tau))dx-\int_{\T^d}G(\phi_\eps(T+\vep))dx +\int_\tau^T\int_{\T^d}  A_{ij}(x)\partial_{j}\phi_\eps g'(\phi_\eps)\partial _i \phi_\eps dxdt\\ & +  \int_\tau^{T+\vep}\int_{\T^d} |\xi_\eps\star(D\tilde \phi)|^r g(\phi_\eps)dx dt \leq \int_\tau^{T+\vep}\int_{\T^d} g(\phi_\eps)\alpha_\eps(t,x) dx dt +\int_\tau^{T+\vep}\int_{\T^d} g(\phi_\eps)R_\eps dxdt\ .
\end{align*}
Since $\int_\tau^T\int_{\T^d}  A_{ij}(x)\partial_{j}\phi_\eps g'(\phi_\eps)\partial _i \phi_\eps dxdt\geq 0$,  and since $\phi_\eps(T+\vep)= \xi^\eps\star\phi_T$, we obtain 
\begin{align*} &\int_{\T^d}G(\phi_\eps(\tau))dx-\int_{\T^d}G(\xi^\eps\star\phi_T)dx +  \int_\tau^{T+\vep}\int_{\T^d} |\xi_\eps\star(D\tilde \phi)|^r g(\phi_\eps)dx dt \\ &  \leq \int_\tau^{T+\vep}\int_{\T^d} g(\phi_\eps)\alpha_\eps(t,x) dx dt +\int_\tau^{T+\vep}\int_{\T^d} g(\phi_\eps)R_\eps dxdt\ .
\end{align*}
Since $g$ is bounded, while $R_\epsilon$ and $\alpha_\vep$ converge in $\pelle r$ and in $\pelle p$ respectively, we can pass  to the limit as $\eps$ goes to zero and for almost every $\tau$ we get (\ref{chain}).
\vskip0.3em
Now we proceed with the desired estimate. First we observe that, up to replacing $g(r)$ with $g(r\wedge k)$, we can assume that $\phi$ is bounded and that $g$ may be any $C^1$ function. In particular, we take $g(\phi)= \phi^{(\sigma-1) r}$ for $\sigma > 1$, obtaining
\begin{align*}
& \frac1{(\sigma-1)r+1}\int_{\T^d}  \phi(\tau)^{(\sigma-1)r+1}  \, dx   +  \frac1{\sigma^r}\int_\tau^T \int_{\T^d} |D\phi^\sigma  |^r\, dxdt 
\\
& \leq  \int_\tau^T \int_{\T^d} \alpha\, \phi^{(\sigma-1)r}\, dxdt + \frac1{(\sigma-1)r+1}\int_{\T^d}   \phi_T^{(\sigma-1)r+1}\,  dx.
\end{align*}
Let us denote henceforth by $c$ possibly different constants only depending on $r$, $\sigma$, $d$ and $T$.  By arbitrariness of $\tau$, the previous inequality implies
\begin{align*}
&  \| \phi^\sigma\|_{\limitate{\frac{(\sigma-1)r+1}\sigma}}^{\frac{(\sigma-1)r+1}\sigma}   +  \|D\phi^\sigma\|_{\pelle r}^r
\\
& \leq  c \int_0^T \int_{\T^d} \alpha\, \phi^{(\sigma-1)r}\, dxdt + c \int_{\T^d}   \phi_T^{(\sigma-1)r+1}\,  dx.
\end{align*}
On the other hand, by interpolation we have (see e.g. \cite[Proposition 3.1, Chapter 1]{DiB})
\be\label{dibe}
 \| v\|_{\pelle q}^q \leq  c\,  \| v\|_{\limitate \eta}^{\frac{\eta r}d}     \|Dv\|_{\pelle r}^r \qquad \hbox{where $q=r \,\frac{d+\eta}d$}
\ee
for any $v\in L^r((0,T);W^{1,r}(\T^d))$ such that $\inte v(t) \, dx=0$ a.e. in $(0,T)$. So we deduce that
\begin{align*}
&  \| \phi^\sigma\|_{\pelle q}^q     \leq  c \left\{\int_0^T \int_{\T^d} \alpha\, \phi^{(\sigma-1)r}\, dxdt +  \int_{\T^d}   \phi_T^{(\sigma-1)r+1}\,  dx\right\}^{1+\frac rd} 
\\
& \qquad + c \int_0^T \left(\inte \phi(t)^\sigma\, dx \right)^q\, dt
\end{align*}
for $\eta=\frac{(\sigma-1)r+1}\sigma$ and  $q= r\,\frac{\eta+d}d$. We choose $\sigma$ such that
$$
\sigma q= (\sigma-1)r p'
$$
and therefore, by H\"older inequality, we conclude
\begin{align*}
&  \| \phi^\sigma\|_{\pelle q}^q     \leq  c \, \| \alpha\|_{\pelle p}^{1+\frac rd} \| \phi^\sigma\|_{\pelle q}^{\frac q{p'}(1+\frac rd)}  + c \left(\int_{\T^d}   \phi_T^{(\sigma-1)r+1}\,  dx\right)^{1+\frac rd} 
\\
& \qquad + c \int_0^T \left(\inte \phi(t)^\sigma\, dx \right)^q\, dt\,.
\end{align*}
Since $ \inte \phi(t)\, dx$ is estimated in terms of $\|\alpha\|_{\pelle1}$ and $\|\phi_T\|_{\elle1}$, last term can be absorbed into the left-hand side up to a constant $C= C( \|\alpha\|_{\pelle1}, \|\phi_T\|_{\elle1})$. 
Moreover, since $p<1+\frac dr$, we have $\frac q{p'}(1+\frac rd)<q$. Hence we end up  with an estimate
$$
\| \phi^\sigma\|_{\pelle q}^q     \leq   C( \|\alpha\|_{\pelle p}, \|\phi_T\|_{\elle{(\sigma-1)r+1} })\,.
$$
Computing the value of $\sigma$ in terms of $r$ and $p$ we get 
$$
q\sigma= \frac{rp(1+d)}{d-r(p-1)}\qquad \hbox{and}\quad (\sigma-1)r+1=\frac{d( r(p-1)+1)}{d-r(p-1)}
$$
so the first part of the Theorem is proved.
\vskip0.3em
Finally, we prove the $L^\infty$ estimate by using a  strategy which goes back  to \cite{St}. To this purpose, we replace $\phi$ with $\phi-k$ and use \eqref{chain} with $g(s)= (s_+)^{r'}$; for any $k\geq \|\phi_T\|_{L^\infty(\T^d)}$ we obtain
$$
 \int_{\T^d}  [(\phi-k)_+(\tau)]^{\sigma+1}  \, dx   +    \int_\tau^T \int_{\T^d} |D(\phi-k)_+^\sigma  |^r\, dxdt 
\leq  \int_\tau^T \int_{\T^d} \alpha\, (\phi-k)_+^{\sigma}\, dxdt 
$$
with $\sigma=r'$. Using as before the embedding \eqref{dibe} we get
\begin{align*}
& \|(\phi-k)_+^\sigma\|_{\pelle q}^q\leq  c \left\{ \int_0^T \int_{\T^d} \alpha\, (\phi-k)_+^{\sigma}\, dxdt \right\}^{1+\frac rd}+ c \int_0^T \left(\inte (\phi-k)_+^\sigma\, dx \right)^q\, dt
\\
& \qquad \leq c\left\{ \int_0^T \int_{\T^d} \alpha\, (\phi-k)_+^{\sigma}\, dxdt \right\}^{1+\frac rd}+
c \int_0^T |\{ x\, : \phi(t)>k\}|^{q-1}\inte (\phi-k)_+^{\sigma\, q}\,dxdt\,,
\end{align*}
where, using that $\sigma=r'$, we have 
\be\label{q}
q= r \, \frac{\frac{\sigma+1}\sigma+d}d= \frac rd (d+2-\frac1r)\,.
\ee
Notice that $|\{ x\, : \phi(t)>k\}|$ is uniformly small provided $k$ is large, only depending on $\|\alpha\|_{L^1}$ and $\|\phi_T\|_{L^1}$. Therefore, absorbing last term in the left-hand side we deduce
$$
  \|(\phi-k)_+^\sigma\|_{\pelle q}^q\leq  c \left\{ \int_0^T \int_{\T^d} \alpha\, (\phi-k)_+^{\sigma}\, dxdt \right\}^{1+\frac rd} 
$$
for some $c=c(\|\alpha\|_{\pelle1}, \|\phi_T\|_{\elle 1})$. One can check that, since $r>1$, \eqref{q} implies   $q>1+\frac rd$ and, in particular, $\frac1q+\frac1p<1$. Thus, by H\"older inequality we get
$$
\|(\phi-k)_+^\sigma\|_{\pelle q}^q\leq c \, \|(\phi-k)_+^\sigma\|_{\pelle q}^{1+\frac rd} \, \|\alpha\|_{\pelle p}^{1+\frac rd}\, |A_k|^{(1-\frac1q-\frac1p)(1+\frac rd)}  \,,
$$
where $A_k:= \{(t,x)\,: \phi(t,x)>k\}$. Since, for any $h>k$ we have 
$$
\intq (\phi-k)_+^{\sigma q}\,dxdt \geq |A_h| (h-k)^{\sigma q}\,,
$$
we end up with the inequality
$$
|A_h|^{1-\frac1q(1+\frac rd)} (h-k)^{\sigma\, q- \sigma(1+\frac rd)}\leq \|(\phi-k)_+^\sigma\|_{\pelle q}^{q-(1+\frac rd)} \leq c \,   \|\alpha\|_{\pelle p}^{1+\frac rd}\, |A_k|^{(1-\frac1q-\frac1p)(1+\frac rd)}
$$
which means that
$$
|A_h|\leq C\, \frac{|A_k|^{\beta}}{(h-k)^\de}\qquad \forall h>k\geq \|\phi_T\|_{L^\infty(\T^d)}
$$
for some $C=C(\|\alpha\|_{\pelle p})$, some $\delta>0$ and with $\beta= \frac{(1-\frac1q-\frac1p)(1+\frac rd)}{1-\frac1q(1+\frac rd)}$. One can check that $\beta>1$ since $p>1+\frac dr$. Therefore, by  a classical iteration lemma  (see e.g. \cite{St}), it follows that $|A_{k_0}|=0$ for some (explicit) $k_0>0$, which in particular implies  the desired bound in terms of $\|\alpha\|_{\pelle p}$ and $\|\phi_T\|_{L^\infty(\T^d)}$.
\end{proof}

\begin{Remark}{\rm
The assumption that $\phi$ is  nonnegative can  be dropped and in this case the estimates are given on $\phi_+$; indeed,   if  $\phi$ satisfies \rife{ineq},  then $\phi_+$ also does. This can be seen in the previous proof  by taking $g=g(r_+)$, with $g(0)=0$.

Let us also stress that the Lipschitz continuity of the matrix $A$ was only used to recover the estimate from the distributional formulation (namely, to be sure that $\phi$ is limit of solutions of smooth approximating problems). The constant $C$ of the estimate, however, does not depend on $A$ in any way;   in particular, the estimate will hold uniformly for any viscous approximation to possibly less regular matrices. }
\end{Remark}

As a corollary, we deduce the following result for problem \eqref{eq:HJ}.

\begin{Theorem} \label{thm:estiHJ}
Assume that   \eqref{HypGrowthH} and \eqref{HypA} hold true and let $\phi$ satisfy \eqref{eq:HJ} with $\alpha\in \pelle p$, $\phi_T\in \elle\infty$. Then, $\phi_+$ satisfies the estimates of Theorem \ref{int-reg}. In particular, if $\phi$ is bounded below, we have
$$
\|\phi\|_{L^\infty((0,T), L^{\eta}(\T^d))}+ \|\phi\|_{L^\gamma((0,T)\times \T^d)}\leq C 
$$
where $\eta=\frac{d( r(p-1)+1)}{d-r(p-1)}$ and $\gamma = \frac{rp(1+d)}{d-r(p-1)}$ if $p<1+\frac dr$ and $\eta=\gamma=+\infty$ if $p>1+\frac dr$, with a constant $C$ depending on $T, p,d,r, C_2, C_3$ (appearing in \eqref{HypGrowthH} and \eqref{HypA})  and on $\|\alpha\|_{L^p((0,T)\times \T^d)}, \|\phi_T\|_{\elle \eta}$ and $\|\phi_-\|_{L^\infty(\T^d)}$. 
\end{Theorem}


\section{Two optimization problems}\label{sec:2opti}

Mean field games systems with local coupling can be studied as an optimality condition between two  problems in duality.

The first optimization problem is described as follows:  let us denote by $\mathcal K_0$ the set of   maps $\phi\in {\mathcal C}^2([0,T]\times \T^d)$ such that $\phi(T,x)= \phi_T(x)$ and define, on $\mathcal K_0$, the functional 
\be\label{DefmathcalA}
{\mathcal A}(\phi)= \int_0^T\int_{\T^d}  F^*\left(x,-\partial_t\phi(t,x)-A_{ij}\partial_{ij}\phi+H(x,D\phi(t,x)) \right)\ dxdt - \int_{\T^d} \phi(0,x)dm_0(x). 
\ee
Then the problem consists in optimizing
\be\label{PB:dual2}
\inf_{\phi\in \mathcal K_0} \mathcal A(\phi)\;.
\ee

For the second optimization problem, let $\mathcal K_1$ be the set of pairs $(m, w)\in L^1((0,T)\times \T^d) \times L^1((0,T)\times \T^d,\R^d)$ such that $m(t,x)\geq 0$ a.e., with $\ds \int_{\T^d}m(t,x)dx=1$ for a.e. $t\in (0,T)$, and which satisfy in the sense of distributions the continuity equation
\be\label{conteq}
\partial_t m-\partial_{ij}(A_{ij}(x)m)+{\rm div} (w)=0\; {\rm in}\; (0,T)\times \T^d, \qquad m(0)=m_0. 
\ee
On the set $\mathcal K_1$, let us  define  the following  functional
$$
{\mathcal B}(m,w)=  \int_0^T\int_{\T^d} m(t,x) H^*\left(x, -\frac{w(t,x)}{m(t,x)}\right)+ F(x,m(t,x)) \ dxdt + \int_{\T^d} \phi_T(x)m(T,x)dx
$$
where, for $m(t,x)=0$, we impose that 
$$
m(t,x) H^*\left(x, -\frac{w(t,x)}{m(t,x)}\right)=\left\{\begin{array}{ll}
+\infty & {\rm if }\; w(t,x)\neq 0\\
0 & {\rm if }\; w(t,x)=0
\end{array}\right..
$$
Since $H^*$ and $F$ are bounded from below and $m\geq 0$ a.e., the first integral in ${\mathcal B}(m,w)$  is well defined in  $\R\cup\{+\infty\}$. In order to give a meaning to the last integral $\ \int_{\T^d} \phi_T(x)m(T,x)dx\ $ we proceed as follows: let us define$\ \ds v(t,x)= -\frac{w(t,x)}{m(t,x)}\ $ if $m(t,x)>0$ and $\ v(t,x)=0\ $ otherwise. 

Thanks to the growth of $H^*$ (implied by  \eqref{HypHstar}, which follows from (H2)), ${\mathcal B}(m,w)$ is infinite if $m|v|^{r'}\notin L^{1}(dxdt)$. Therefore, we can assume without loss of generality that $m|v|^{r'}\in L^{1}(dxdt)$, or, equivalently, that   $v\in L^{r'}(m\ dxdt)$. In this case equation \eqref{conteq} can be rewritten as a Kolmogorov equation 
\be\label{conteq12}
\partial_t m-\partial _{ij}(A_{ij}(x)m)-{\rm div} (mv)=0\; {\rm in}\; (0,T)\times \T^d, \qquad m(0)=m_0. 
\ee

\begin{Lemma}  \label{lem:m-cts}
The map $t\mapsto m(t)$ is H\"older continuous a.e. for the weak* topology of $P(\T^d)$. 
\end{Lemma}
This Lemma implies,  in particular, that the measure $m(t)$ is defined for any $t$, therefore the second integral term in the definition of ${\mathcal B}(m,w)$ is well defined. 

For the sake of completeness, we give the proof here.

\begin{proof} We first extend the pairs $(m,w)$  to $[-1,T]\times \T^d$ by defining $m=m_0$ on $[-1,0]$ and $w(s,x)=0$ for $(s,x)\in (-1,0)\times \T^d$. Note that  $\partial_t m-\partial_{ij}(\tilde A_{ij}(t,x)m)+{\rm div}(w)=0$ holds in the sense of distributions on $(-1,T)\times \T^d$, where $\tilde A_{ij}(t,x)=A_{ij}(x)$ if $t\in (0,T)$ and $\tilde A_{ij}(t,x)=0$ otherwise.   Let $\xi$ be a standard convolution kernel in $(t,x)$, a support compact on $\R^{d+1}$ and $\xi^\eps(t,x)=\xi((t,x)/\eps)/(\eps)^{d+1}$, $\xi^\eps\geq0,\ \int_{\R^{d+1}}\xi^\eps(t,x) dtdx=1$, for all $\eps>0$. Let $m_\eps= \xi^\eps\star m$ and $w_\eps=\xi^\eps\star w$. Then  $m_\eps, w_\eps$ are $C^\infty$ and $\int_{\T^d}m_\eps(t,x)dx=1$ for all $t\in(0,T)$ and $\eps>0$ small enough.

Convolving $\xi^\eps$ with (\ref{conteq}), we obtain
$$
\partial_t m_\eps-\partial_{ij}(\xi^\eps*(\tilde A_{ij}(t,x) m))+{\rm div} (w_\eps)=0\; {\rm in}\; (-1/2,T)\times \T^d,
$$
with 
$$
m_\eps(-1/2,x)=\int_\R\int_{\T^d} \xi^\eps(s, x-y) m_0(y)dyds.
$$
The equation can be rewritten as 
\be\label{dual-eps}
\partial_t m_\eps-\partial_{ij}( \tilde A^\eps_{ij}(t,x)m_\eps))-{\rm div} (m_\eps v_\eps)=0\; {\rm in}\; (-1/2,T)\times \T^d
\ee
where 
$ \tilde A^\eps_{ij}=\frac{\xi^\eps\star(\tilde A_{ij}m)}{m_\eps}$ and $v_\eps=-\frac{w_\eps}{m_\eps}$.

Let us consider the following stochastic differential equations defined for all $\eps>0$
\be\label{eqstoch}
\left\{\begin{array}{ll} dX^\eps_t=v_\eps(t,X^\eps_t)dt+\Sigma_\eps(X^\eps_t) dB^\eps_t & t\in[-1/2,T]\\ X^\eps_{-1/2}=Z_{-1/2}^\eps \end{array}\right.,
\ee
where $dB^\eps_t$ is a standard $d$-dimensional Brownian motion over some probability space $(\Omega,\mathcal A,\mathbb P)$, $\Sigma_\eps\Sigma_\eps^T=\tilde A^\eps$, and the initial condition $Z^\eps_{-1/2}\in L^1(\T^d)$ is random, independent of $(B^\eps_t)$ and with law $m_\eps(-1/2,\cdot)$. 

For all $\eps>0$, the vector field $v_\eps$ is continuous, uniformly Lipschitz continuous in space and bounded. Therefore, there exists a unique solution to (\ref{eqstoch}). Moreover, as a consequence of Ito's formula, we have that, if the density $\mathcal L (Z^\eps_0)=\xi^\eps\star m_0$, then $m_\eps(t)=\mathcal L(X^\eps_t)$ solves (\ref{dual-eps}) in the sense of distributions.

Let $\bf d_1$ be the Kantorovich-Rubinstein distance on $P(\T^d)$ and $\gamma_\eps\in\Pi(m_\eps(t),m_\eps(s))$ the law of the pair $(X^\eps_t, X^\eps_s)$ for $0\leq s<t\leq T$, where $\Pi(m_\eps(t),m_\eps(s))$ is the set of Borel probability measures $\mu$ on $\T^d\times\T^d$ such that $\mu(A\times\T^d)=m_\eps(t,A)$ and $\mu(\T^d\times A)=m_\eps(s,A)$ for any Borel set $A\in \T^d$. We have
$$
{\bf d_1 }(m_\eps(t),m_\eps(s))\leq \int_{\T^d\times\T^d}|x-y| d\gamma_\eps(x,y)= \mathbb E[|X^\eps_t-X^\eps_s|].
$$ 
Moreover,
\begin{align*}
\mathbb E [|X^\eps_t-X^\eps_s|]&\leq \mathbb E[\int_s^t |v_\eps(\tau,X^\eps_\tau)|d\tau]+\mathbb E\left[\left|\int_s^t \Sigma_\eps(X^\ep_\tau)dB_\tau\right|\right]\\
&\leq \int_s^t\int_{\T^d}|v_\eps(\tau,x)| m_\eps(\tau,x)dx d\tau+\left(\mathbb E\left[\int_s^t \Sigma_\eps\Sigma_\eps^*(X^\ep_\tau)d\tau\right]\right)^{1/2}\\
&\leq \int_s^t\int_{\T^d}|v_\eps(\tau,x)| m_\eps(\tau,x)dx d\tau+\| A\|_\infty C |t-s|^\frac{1}{2}.
\end{align*}
Recalling the definition of $v_\eps$, we have  that $m_\eps |v_\eps|^{r'}= \frac{|w_\eps|^{r'}}{m_\eps^{r'-1}}$ belongs to $L^1([0,T]\times\T^d)$ for all $\eps>0$. Indeed, the function $(m,w)\mapsto \frac{|w|^{r'}}{m^{r'-1}}$ is convex and $\frac{|w|^{r'}}{m^{r'-1}}$ belongs to $L^1([0,T]\times\T^d)$. Thus 
$$
\int_{0}^{T }\int_{\T^d} \frac{|\xi^\eps \star w|^{r'}}{(\xi^\eps\star m)^{r'-1}}  dx d\tau\leq\int_{0}^{T} \int_{\T^d} \xi^\eps\star \left(\frac{|w|^{r'}}{m^{r'-1}}\right) dx d\tau\leq \left\|\frac{|w|^{r'}}{m^{r'-1}} \right\|_1.$$ 

Therefore, using H\"older inequality, 
\begin{align*}
{\bf d_1 }(m_\eps(t),m_\eps(s))
&\leq \left(\int_s^t\int_{\T^d}|v_\eps(\tau,x)| ^{r'}m_\eps(\tau,x)dx d\tau\right)^{\frac{1}{r'}}\left(\int_s^t\int_{\T^d} m_{\eps}(\tau,x)dxd\tau\right)^\frac{1}{r}+\|A\|_\infty C |t-s|^\frac{1}{2}\\
&\leq  \left\|\frac{|w|^{r'}}{m^{r'-1}} \right\|_1^{\frac{1}{r'}} |t-s|^{\frac{1}{r}}+\|A\|_\infty |t-s|^\frac{1}{2}. 
\end{align*}
Letting $\eps\to 0$ we have $m_\eps\to m$ in $L^1([0,T]\times\T^d)$ and for a.e. $\tau \in [0,T],$ $ m_\eps(\tau )\to m(\tau)$ in $L^1(\T^d)$, moreover for a.e. $0\leq s<t\leq T$
$$
\lim_{\eps\to 0 } {\bf d_1 }(m_\eps(t),m_\eps(s))={\bf d_1 }(m(t),m(s)).
$$ 
Thus for a.e. $0\leq s<t\leq T$
$$
{\bf d_1 }(m(t),m(s))\leq C |t-s|^{\frac{1}{r}} + \|A\|_\infty |t-s|^\frac{1}{2}.
$$ 
\end{proof}

The second optimal control problem is the following: 
\be\label{Pb:mw2}
\inf_{(m,w)\in \mathcal K_1} \mathcal B(m,w)\;.
\ee

\begin{Lemma}\label{Lem:dualite} We have
$$
\inf_{\phi\in \mathcal K_0}{\mathcal A}(\phi) = - \min_{(m,w)\in \mathcal K_1} {\mathcal B}(m,w). 
$$
Moreover, the minimum in the right-hand side is achieved by a unique pair $(m,w)\in \mathcal K_1$ satisfying $(m,w)\in  L^q((0,T)\times \T^d)\times L^{\frac{r'q}{r'+q-1}}((0,T)\times \T^d)$. 
\end{Lemma}

\begin{Remark}{\rm Note that $\frac{r'q}{r'+q-1}>1$ because $r'>1$ and $q>1$. 
}\end{Remark}

\begin{proof}
The strategy of proof---which is very close to the corresponding one in \cite{c1, cg, CCN}---consists in applying the Fenchel-Rockafellar duality theorem (cf. e.g., \cite{ET}). In order to do so, it is better to reformulate the first optimization problem \eqref{PB:dual2}  in a more suitable form. 
Let  $E_0= {\mathcal C}^2([0,T]\times \T^d)$ and  $E_1=   {\mathcal C}^0([0,T]\times \T^d, \R)\times {\mathcal C}^0([0,T]\times \T^d, \R^d)$. We define on $E_0$  the functional 
$$
{\mathcal F}(\phi)= -\int_{\T^d} m_0(x)\phi(0,x)dx +\chi_S(\phi),  
$$
where $\chi_S$ is the characteristic function of the set $S=\{\phi\in E_0, \; \phi(T, \cdot)= \phi_T\}$, i.e., $\chi_S(\phi)=0$ if $\phi\in S$ and $+\infty$ otherwise. For $(a,b)\in E_1$, we define
$$
{\mathcal G} (a,b)= \int_0^T\int_{\T^d}  F^*(x,-a(t,x)+H(x,b(t,x)) )\ dxdt \;.
$$
The functional  ${\mathcal F}$ is convex and lower semi-continuous on $E_0$ while ${\mathcal G}$ is convex  and continuous on $E_1$. Let $\Lambda:E_0\to E_1$ be the bounded linear operator defined by $\ds \Lambda (\phi)= (\partial_t\phi+A_{ij}\partial_{ij}\phi , D\phi)$. We can observe that 
$$
\inf_{\phi\in \mathcal K_0} \mathcal A(\phi)= \inf_{\phi\in E_0} \left\{ {\mathcal F}(\phi)+{\mathcal G}(\Lambda(\phi))\right\}.
$$
It is easy to verify that the qualification hypothesis, that ensures the stability of the above optimization problem, holds. Indeed, there is a map $\phi$ such that ${\mathcal F}(\phi)<+\infty$ and 
such that ${\mathcal G} $ is continuous at $\Lambda (\phi)$: it is enough to take $\phi(t,x)=\phi_T(x)$. 


Therefore we can apply the  Fenchel-Rockafellar duality theorem, which states that 
$$
\inf_{\phi\in E_0} \left\{ {\mathcal F}(\phi)+{\mathcal G}(\Lambda(\phi))\right\}
=
\max_{ (m,w)\in E_1'} \left\{ -  {\mathcal F}^*(\Lambda^*(m,w))-{\mathcal G}^*(-(m,w))\right\}
$$
where $E_1'$ is the dual space of $E_1$, i.e., the set of vector valued Radon measures $(m, w)$ over $[0,T]\times \T^d$ with values in $\R\times \R^d$,  $E_0'$ is the dual space of $E_0$, $\Lambda^*: E'_1\to E'_0$ is the dual operator of $\Lambda$ and ${\mathcal F}^*$ and ${\mathcal G}^*$ are the convex conjugates of ${\mathcal F}$ and ${\mathcal G}$ respectively. By a direct computation we have 
$$
{\mathcal F}^*(\Lambda^*(m,w))
= \left\{ \begin{array}{ll}
\ds \int_{\T^d} \phi_T(x)dm(T,x) & {\rm if }\; \partial_t m-\partial_{ij}(A_{ij}m)+{\rm div} (w)=0, \; m(0)=m_0\\
+\infty & {\rm otherwise}
\end{array}\right.
$$
where the equation $\ \partial_t m-\partial_{ij}(A_{ij}m)+{\rm div} (w)=0, \; m(0)=m_0\ $ holds in the sense of distributions. 
Following \cite{c1}, we have ${\mathcal G}^*(m,w)=+\infty$ if $(m,w)\notin L^1$ and, if $(m,w)\in L^1$, 
$$
{\mathcal G}^*(m,w)= \int_0^T\int_{\T^d} K^*(x,m(t,x),w(t,x))dtdx ,
$$
where $$K^*(x,m,w)=\left\{\begin{array}{ll}F(x,-m)-mH^*(x,-\frac{w}{m})& \text{ if } m<0 \\
0& \text{ if } m=0,w=0\\
+\infty & \text{ otherwhise} \end{array}\right.$$ is the convex conjugate of  $$ K (x,a,b)=F^*(x,-a+H(x,b))\quad \forall (x,a,b)\in\T^d\times\R\times\R^d.$$
Therefore
$$
\begin{array}{l}
\ds \max_{ (m,w)\in E_1'} \left\{ -  {\mathcal F}^*(\Lambda^*(m,w))-{\mathcal G^*}(-(m,w))\right\}\\
\qquad \qquad \qquad \ds = \max \left\{ \int_0^T\int_{\T^d}  -F(x,m ) -m H^*( x, -\frac{w}{m})\ dtdx  -\int_{\T^d} \phi_T(x) m(T,x)\ dx \right\}
 \end{array}
$$
where the last maximum is taken over the $L^1$ maps $(m,w)$ such that $m\geq 0$ a.e. and 
$$
\partial_t m-\partial_{ij}(A_{ij}m)+{\rm div} (w)=0,\; m(0)=m_0\,
$$
holds in the sense of distributions.
Since $\ds\ \int_{\T^d} m_0=1\ $, it follows that $\ds\ \int_{\T^d} m(t)=1\ $ for any $t\in [0,T]$. Thus the pair $(m,w)$ belongs to the set $\mathcal K_1$ and the first part of the statement is proved.

Take now   an optimal $(m, w)\in \mathcal K_1$ in the above system. Observe that due to optimality  we have $w(t,x)=0$ for all $(t,x)\in [0,T]\times\T^d$ such that $m(t,x)=0$. The growth conditions \eqref{HypGrowthH} and \eqref{HypGrowthF} imply 
$$
\begin{array}{rl}
 C \; \geq & \ds  \int_0^T\int_{\T^d}  F(x,m ) +m H^*( x, -\frac{w}{m})\ dtdx  +\int_{\T^d} \phi_T(x) m(T,x)\ dx \\
  \geq & \ds   \int_0^T\int_{\T^d}  \left(\frac{1}{ C}|m|^{q}+
 \frac{m}{ C } \left|\frac{w}{m}\right|^{r'} - C(m+1) \right) dxdt -\|\phi_T\|_\infty .
 \end{array}
 $$
Therefore $m\in L^q$.  Moreover, by H\"older inequality, we also have 
$$
\int_0^T\int_{\T^d} |w|^{\frac{r'q}{r'+q-1}} = \int\int_{\{m>0\}} |w|^{\frac{r'q}{r'+q-1}} \leq \|m\|_q^{\frac{r'-1}{r'+q-1}} \left(\int\int_{\{m>0\}} \frac{|w|^{r'}}{m^{r'-1}} \right)^{\frac{q}{r'+q-1}} \leq C
$$
so that $w\in L^{\frac{r'q}{r'+q-1}}$. Finally, a minimizer to \eqref{Pb:mw2} should be unique, because the set $\mathcal K_1$ is convex and the maps $F(x,\cdot)$ and $H^*(x,\cdot)$  are strictly convex: thus $m$ is unique and so is $\ds \frac{w}{m}$ in $\{m>0\}$. As $w=0$ in $\{m=0\}$, uniqueness of $w$ follows as well. 
\end{proof}

\section{Analysis of the optimal control of the HJ equation}\label{sec:optiHJ}

In general, we do not expect problem \eqref{PB:dual2} to have a solution. In this section we exhibit a relaxation for  \eqref{PB:dual2}  (Proposition \ref{prop:valeursegales}) and show that this relaxed problem has at least one solution (Proposition \ref{Prop:existence}).

\subsection{The relaxed problem}

Recall that the exponents $\eta>1$ and  $\gamma>1$ are defined in Theorem \ref{thm:estiHJ}. 
Let ${\mathcal K}$ be the set of pairs $(\phi,\alpha)\in L^\gamma((0,T)\times \T^d)\times L^p((0,T)\times \T^d)$ such that $D\phi\in L^r((0,T)\times \T^d)$ and which satisfy in the sense of distributions
\be\label{ineq:phi}
-\partial_t \phi-A_{ij}(x)\partial_{ij}\phi+H(x,D\phi) \leq \alpha, \qquad \phi(T,\cdot)\leq \phi_T
\ee
(for the precise meaning of the inequality, see the beginning of Section \ref{sec:esti}). 
The following statement explains that $\phi$ has a ``trace" in a weak sense. 
\begin{Lemma}\label{lem:intem0phiBV} Let $(\phi,\alpha)\in {\mathcal K}$. Then, for any Lipschitz continuous map $\zeta: \T^d\to \R$, the map $\ds t\to \int_{\T^d} \zeta(x)\phi(t,x)dx$ has a BV representative on $[0,T]$. Moreover, if we denote by $\ds \int_{\T^d} \zeta(x)\phi(t^+,x)dx$ its right limit at $t\in [0,T)$, then 
the map $\ds \zeta\to \int_{\T^d} \zeta(x)\phi(t^+,x)dx$ is continuous in $L^{\eta'}(\T^d)$. 
\end{Lemma}

As a consequence,  for any nonnegative $C^1$ map $\vartheta:[0,T]\times \T^d\to \R$, one can write the integration by parts formula: for any $0\leq t_1\leq t_2\leq T$, 
$$
-\left[\int_{\T^d}\vartheta\phi\right]_{t_1}^{t_2} + \int_{t_1}^{t_2}\int_{\T^d} \phi\partial_t\vartheta+ \lg D\vartheta, AD\phi\rg + \vartheta(\partial_iA_{ij}\partial_j\phi+ H(x,D\phi)) \leq \int_{t_1}^{t_2}\int_{\T^d} \alpha \vartheta.
$$

\begin{proof}[Proof of Lemma \ref{lem:intem0phiBV}.] One easily checks that, for any Lipschitz continuous, nonnegative map $\zeta: \T^d\to \R$, 
$$
- \frac{d}{dt} \int_{\T^d} \zeta \phi(t)+ \int_{\T^d} \lg D\zeta,AD\phi(t)\rg +\zeta (\partial_iA_{ij}\partial_j\phi+ H(x,D\phi) -\alpha ) \leq 0, 
$$
holds in the sense of distributions. As the second integral is in $L^1((0,T))$, the map $t\to \int_{\T^d} \zeta \phi(t)$ is BV. If now $\zeta$ is Lipschitz continuous and changes sign, one can write $\zeta=\zeta^+-\zeta^-$ and the map $t\to \int_{\T^d} \zeta \phi(t)=  \int_{\T^d} \zeta^+ \phi(t)-  \int_{\T^d} \zeta^- \phi(t)$ is still BV. The continuity with respect to $\zeta$ comes from the $L^\infty((0,T), L^\eta(\T^d))$ estimate on $\phi$ given in Theorem \ref{thm:estiHJ}.  
\end{proof}

We extend the functional $\mathcal A$ to  ${\mathcal K}$ by setting 
$$
\mathcal A(\phi,\alpha) = \int_0^T\int_{\T^d}  F^*(x,\alpha(x,t))\ dxdt - \int_{\T^d} \phi(x,0)m_0(x)\ dx\qquad \forall (\phi, \alpha)\in {\mathcal K}.
$$
The next proposition explains that the problem
\be\label{PB:dual-relaxed}
\inf_{(\phi, \alpha)\in {\mathcal K}} \mathcal A(\phi,\alpha)
\ee
is the relaxed problem of \eqref{PB:dual2}. For this we first note that
\be\label{eq:positif}
\inf_{(\phi, \alpha)\in {\mathcal K}} \mathcal A(\phi,\alpha)= \inf_{(\phi, \alpha)\in {\mathcal K}, \ \alpha\geq0\ {\rm a.e.}} \mathcal A(\phi,\alpha)
\ee
because one can always replace $\alpha$ by $\alpha\vee0$ since $F^*(x,\alpha)= 0$ for $\alpha\leq0$. 

\begin{Proposition}\label{prop:valeursegales} We have
$$
 \inf_{\phi\in \mathcal K_0} \mathcal A(\phi)= \inf_{(\phi, \alpha)\in {\mathcal K}} \mathcal A(\phi,\alpha).
 $$
\end{Proposition}

The proof requires the following inequality: 

\begin{Lemma}\label{lem:phialphamw} Let $(\phi,\alpha)\in {\mathcal K}$ and $(m,w)\in {\mathcal K}_1$. Assume that $mH^*(\cdot, -w/m)\in L^1((0,T)\times \T^d)$ and $m \in \pelle q$. Then 
\be\label{ineq:ineqfirst}
\left[ \int_{\T^d} m \phi\right]_t^{T} + \int_t^{T}\int_{\T^d} m \left(\alpha  + H^*(x,-\frac{w}{m})\right) \; \geq \; 0
\ee
and 
$$
\left[ \int_{\T^d} m \phi\right]_0^{t} + \int_0^{t}\int_{\T^d} m \left(\alpha  + H^*(x,-\frac{w}{m})\right) \; \geq \; 0.
$$
Moreover, if equality holds in the inequality \eqref{ineq:ineqfirst} for $t=0$, then $w= -mD_pH(x,D\phi)$ a.e. 
\end{Lemma}

\begin{proof} We first extend the pairs $(m,w)$  to $[-1,T+1]\times \T^d$ by defining $m=m_0$ on $[-1,0]$, $m= m(T)$ on $[T,T+1]$ and $w(s,x)=0$ for $(s,x)\in (-1,0)\cup (T,T+1)\times \T^d$. Note that  $\partial_t m-\partial_{ij}(\tilde A_{ij}(t,x)m)+{\rm div}(w)=0$ on $(-1,T+1)\times \T^d$, where $\tilde A_{ij}(t,x)=A_{ij}(x)$ if $t\in (0,T)$ and $\tilde A_{ij}(t,x)=0$ otherwise.  
Let $\xi^\ep= \xi^\ep(t,x)$ be a smooth convolution kernel with support in $B_\ep$; we smoothen the pair $(m,w)$ in a standard way into $(m_\ep,w_\ep)$.
Then $(m_\ep,w_\ep)$ solves 
\begin{equation}
\label{eq:mep}
\partial_t m_\ep-\partial_{ij} (\tilde A_{ij} m_\ep)+{\rm div}(w_\ep)= \partial_i R_\epsilon \qquad {\rm in}\; (-1/2, T+1/2)
\end{equation}
in the sense of distributions, where
\begin{equation}
R_\epsilon := [\xi^\epsilon,\partial_j \tilde A_{ij}](m) + [\xi^\epsilon,\tilde A_{ij}\partial_j](m).
\end{equation}
Here we use again the  commutator notation (cf. \cite{DL89})
\begin{equation} \label{eq:commutators}
[\xi^\epsilon,c](f) := \xi^\epsilon \star (cf) - c(\xi^\epsilon \star f)\,.
\end{equation}
Invoking \cite[Lemma II.1]{DL89}, we have that $R_\epsilon \to 0$ in $L^q$, since $m \in L^q$ and $\tilde A_{ij} \in W^{1,\infty}$.

Let us fix  time  $t\in (0,T)$ at which $\phi(t^+)=\phi(t^-)= \phi(t)$ in $L^\gamma(\T)$ and $m_\ep(t)$ converges to $m(t)$. 
By the inequality satisfied by $(\phi,\alpha)$, we have  
$$
\ds \int_t^{T}\int_{\T^d} \phi \partial_t m_\ep +\partial_i \phi \partial_{j}( \tilde A_{ij}m_\ep)+ m_\epsilon H(x,D\phi) + \int_{\T^d} m_\ep(t) \phi(t) - m_\ep(T)\phi_T\\
\leq \int_t^{T}\int_{\T^d} \alpha m_\ep\,.
$$
By \eqref{eq:mep} we have 
$$
 \int_t^{T}\int_{\T^d} \phi \partial_t m_\ep +\partial_i \phi \partial_{j}( \tilde A_{ij}m_\ep)=  \int_t^{T}\int_{\T^d} -\partial_i \phi R_\epsilon + \lg D\phi, w_\ep\rg. 
 $$
On the other hand, by convexity of $H$,  
\be\label{ineq:ineqCV}
\begin{array}{rl}
\ds \int_t^{T}\int_{\T^d} - m_\ep H^*(x,-\frac{w_\ep}{m_\ep}) \; 
\leq & \ds \int_t^{T}\int_{\T^d} \lg w_\ep, D\phi\rg + m_\ep H(x,D\phi) \,.
\end{array}
\ee
Collecting the above (in)equalities we obtain  
$$
\int_{\T^d} m_\ep(t) \phi(t)\leq \inte m_\ep(T)\phi_T +  \int_t^T \int_{\T^d}  m_\ep (\alpha+ H^*(x,-\frac{w_\ep}{m_\ep})) + \partial_j \phi R_\epsilon\, .
$$
By assumption \eqref{CondCroiss} which states that $r\geq p$, and since $D\phi \in L^r$, we have $\ds \iint \partial_j \phi R_\epsilon\to 0$ as $\ep\to 0$.
Following the proof of Lemma 2.7 in \cite{cg} we have
$$
 \int_t^{T}\int_{\T^d} - m_\ep H^*(x,-\frac{w_\ep}{m_\ep}) \to  \int_t^{T}\int_{\T^d} - m H^*(x,-\frac{w}{m}) \qquad 
\mbox{\rm as $\ep\to 0$. }
$$
The continuity of $t\to m(t)$ in $P(\T^d)$ given by Lemma \ref{lem:m-cts} implies the convergence
$$
\int_{\T^d} m_\ep(T) \phi_T\to \int_{\T^d} m(T) \phi_T\,.
$$
Recalling that $\phi$ is bounded below, we finally get by Fatou's  Lemma the inequality
$$
-\|\phi_-\|_\infty+ \int_{\T^d} m(t) (\phi(t)+ \|\phi_-\|_\infty) \leq \inte m(T)\phi_T + \int_t^T \inte m(\alpha+H^*(x, -\frac{w}{m})), 
$$
which implies that  $m(t) \phi(t)$ is integrable with
$$
\inte m(t) \phi(t)\leq \inte m(T)\phi_T + \int_t^T \inte m(\alpha+H^*(x, -\frac{w}{m}))
$$
We can argue similarly in  the time interval $[0,t]$ using that $\inte m_\ep(t) \phi(t) \to \inte m(t) \phi(t)$; this is certainly true, up to  a subsequence,  for a.e. $t$, because $m_\ep \phi$ strongly converges in $\pelle1$ since $\phi\in \pelle \gamma$, $m\in \pelle q$ and $\gamma\geq p$. We obtain then  
$$
\int_{\T^d} m_0 \phi(0)\leq \inte m(t)\phi(t)+  \int_0^t \int_{\T^d}  m (\alpha+ H^*(x,-\frac{w}{m})) .
$$
Let us assume finally that the following equality holds: 
$$
\left[ \int_{\T^d} m \phi\right]_0^{T} + \int_0^{T}\int_{\T^d} m \left(\alpha  + H^*(x,-\frac{w}{m})\right) \; = \; 0. 
$$
Then there is an equality in inequality \eqref{ineq:ineqfirst} for almost all $t$. 
Fix such a $t\in (0,T)$ and let 
$$
E_\sigma(t):= \left\{(s,y)\;, \; s\in [t,T], \; m (H^*(y,-\frac{w}{m}) + H(x,D\phi)) \geq - \lg w, D\phi\rg +\sigma\right\}.
$$ If $|E_\sigma(t)|>0$, then 
for $\ep>0$ small enough, the set 
$$
E_{\ep,\sigma}(t):= \{(s,y)\;, \; s\in [t,T], \; m_\eps (H^*(y,-\frac{w_\ep}{m_\ep}) + H(x,D\phi)) \geq -  \lg w_\ep, D\phi\rg +\sigma/2\}
$$ has a measure larger than $|E_{\sigma}(t)|/2$. Coming back to inequality \eqref{ineq:ineqCV}, we have 
$$
\begin{array}{rl}
\ds \int_t^{T}\int_{\T^d} - m_\ep H^*(x,-\frac{w_\ep}{m_\ep}) \; 
\leq & \ds \int_t^{T}\int_{\T^d} \lg w_\ep, D\phi\rg + m_\ep H(x,D\phi) -|E_{\sigma}(t)|\sigma/4
\end{array}
$$
Then inequality \eqref{ineq:ineqfirst} becomes 
$$
\inte m(t) \phi(t)\leq \inte m(T)\phi_T + \int_t^T \inte m(\alpha+H^*(x, -\frac{w}{m})) -|E_{\sigma}(t)|\sigma/4,
$$
which contradicts the fact that there is an equality in \eqref{ineq:ineqfirst}. So $|E_\sigma(t)|=0$ for any $\sigma$ and for a.e. $t$, which shows that 
$m (H^*(y,-\frac{w}{m}) +H(x,D\phi)) =-  \lg w, D\phi\rg$ a.e. Thus $w= -m D_pH(x,D\phi)$ holds a.e. in $\{m>0\}$ and, as $w=0$ in $\{m=0\}$, a.e. in $(0,T)\times \T^d$. 
\end{proof}

\begin{proof}[Proof of Proposition \ref{prop:valeursegales}] We follow the argument developed by Graber in \cite{Gra14}. 
 Inequality $\ \ds  \inf_{\phi\in \mathcal K_0} \mathcal A(\phi)\geq \inf_{(\phi, \alpha)\in {\mathcal K}} \mathcal A(\phi,\alpha)\ $ being obvious, let us check the reverse one. Let $(\phi,\alpha)\in {\mathcal K}$. For any $(m, w)\in  {\mathcal K}_1$ with $mH^*(\cdot, -\frac{w}{m})\in L^1$, we have, by Lemma \ref{lem:phialphamw}, 
 $$
 \begin{array}{rl}
 \ds   \mathcal A(\phi,\alpha) \; \geq & \ds \int_0^T \int_{\T^d} \alpha m - F(m) - \inte m_0\phi(0) \\
 \geq & \ds \int_0^T \inte -m H^*(x, -\frac{w}{m})- F(m)- \inte m(T)\phi_T = -\mathcal B(m,w)
\end{array}
$$
 Taking the sup with respect to $(m,w)$ in the right-hand side  we obtain thanks to Lemma \ref{Lem:dualite}:
 $$
 \mathcal A(\phi,\alpha) \; \geq \; - \inf_{(m, w)\in {\mathcal K}_1} \mathcal B(m,w) = \inf_{\phi\in \mathcal K_0} \mathcal A(\phi).
 $$
\end{proof}

\subsection{Existence of a solution for the relaxed problem}

The next proposition explains the interest of considering the relaxed problem \eqref{PB:dual-relaxed} instead of the original one \eqref{PB:dual2}.

\begin{Proposition}\label{Prop:existence} The relaxed problem \eqref{PB:dual-relaxed} has at least one solution $(\phi, \alpha) \in {\mathcal K}$ which is bounded below by a constant depending on $\|\phi_T\|_{C^2}$, on $\|A_{ij}\|_{C^0}$ and on $\|H(\cdot, D\phi_T)\|_\infty$.
\end{Proposition}

\begin{proof}  We start with the construction of a suitable minimizing sequence. Let $(\tilde \phi_n)$ be a minimizing sequence for problem \eqref{PB:dual2} and let us set 
\be\label{defalphan}
\alpha_n(t,x)= \max\{0\ ; \ -\partial_t\tilde \phi_n(t,x) -A_{ij}\partial_{ij}\tilde \phi_n(t,x)+H(x,D\tilde \phi_n(t,x))\}.
\ee 
By Proposition \ref{prop:valeursegales} and the fact that $F^*(x,\alpha)=0$ if $\alpha\leq0$, the pair $(\phi_n, \alpha_n)$ is also a minimizing sequence of \eqref{PB:dual-relaxed}. Let $\psi$ be the unique viscosity solution to 
$$
-\partial_t \psi-A_{ij}(x)\partial_{ij}\psi+H(x,D\psi)= 0, \qquad \psi(T,\cdot)= \phi_T.
$$
As $\phi_T$ is $C^2$, $\psi(t,x)\geq \tilde \phi_T(x)-C(T-t)$, where the constant $C$ depends on $\|\phi_T\|_{C^2}$, on $\|A_{ij}\|_{C^0}$ and on $\|H(\cdot, D\phi_T)\|_\infty$. Let $\phi_n$ be the (continuous) viscosity solution to 
\be\label{eqphin}
-\partial_t \phi_n-A_{ij}(x)\partial_{ij}\phi_n+H(x,D\phi_n)\leq \alpha_n, \qquad \psi(T,\cdot)\leq \phi_T.
\ee
By comparison, $\phi_n\geq \tilde \phi_n\vee \psi$.  As $H$ is convex, \eqref{eqphin} holds in the sense of distributions (see  \cite{Ish95}). Therefore the sequence  $(\phi_n,\alpha_n)$ is still minimizing, with the following  bound below for $(\phi_n)$: 
\be\label{unifboundbelow}
\phi_n(t, x)\geq \phi_T(x)-C(T-t).
\ee

\noindent {\bf Step 1:} We claim that $(\alpha_n)$ is bounded in $L^p((0,T)\times \T^d)$. For this, we integrate \eqref{eqphin} against $m_0$ on  $(0,T)\times \T^d$
$$
\int_{\T^d} \phi_n(0)m_0 +\int_0^T\int_{\T^d} \partial_i m_0A_{ij}\partial_j \phi_n + (\partial_j A_{ij}) m_0\partial_j \phi_n +m_0H(x,D\phi_n) \leq \int_0^T \inte m_0\alpha_n+\inte \phi_Tm_0.
$$
As $(1/C_0)\leq m_0\leq C_0$ for some $C_0>0$, $\|Dm_0\|_\infty<+\infty$ and $H$ is coercive, we get
\be\label{phi(0)Dphi}
\int_{\T^d} \phi_n(0)m_0 +\frac{1}{C} \int_0^T\inte |D\phi_n|^r  \leq C_0\|\alpha_n\|_p+C.
\ee
On the other hand, as $(\phi_n)$ is a minimizing sequence and $F^*$ is coercive, 
$$
\frac{1}{C} \|\alpha_n\|_p^p -\inte \phi_n(0)m_0 \leq \int_0^T \inte F^*(x,\alpha_n)-\inte \phi_n(0)m_0+C \leq C. 
$$
Adding the previous inequalities, we get 
$$
\frac{1}{C} \|\alpha_n\|_p^p+\frac{1}{C} \int_0^T\inte |D\phi_n|^r  \leq C_0\|\alpha_n\|_p+C,
$$
so that $(\alpha_n)$ is bounded in $L^p((0,T)\times \T^d)$ while $(D\phi_n)$ is bounded  in $L^r$.\\

\noindent{\bf Step 2:} We show here that $(\phi_n,\alpha_n)$ has a limit. As $(\alpha_n)$ is bounded in $L^p$ and $(\phi_n)$ is uniformly bounded below thanks to \eqref{unifboundbelow}, Theorem \ref{thm:estiHJ} implies that $(\phi_n)$ is bounded in $L^\gamma$. So we can assume with loss of generality that $\alpha_n \rightharpoonup \bar\alpha$ in $L^p$, $\phi_n \rightharpoonup \bar \phi$ in $L^\gamma$ and $D\phi_n \rightharpoonup D\bar\phi$ in $L^r$ where, in view of the convexity of $H$, the pair $(\bar\phi,\bar\alpha)$ belongs to  ${\mathcal K}$. \\

\noindent{\bf Step 3:} We now prove that $(\bar\phi,\bar\alpha)$ is a minimizer. By weak lower semicontinuity arguments, we have
$$
\liminf_n \int_0^T\inte F^*(x,\alpha_n)\geq \int_0^T\inte F^*(x,\bar\alpha).
$$
Let $\ds \zeta_n(t)= \inte m_0\phi_n(t)$ and $\ds \bar \zeta(t)= \inte m_0\bar \phi(t)$. Then $(\zeta_n)$ converges  weak* to $\bar\zeta$ in $L^\infty$ thanks to Theorem \ref{thm:estiHJ}. As 
$$
- \frac{d}{dt} \zeta_n(t)+ \int_{\T^d} \lg Dm_0,AD\phi_n(t)\rg +m_0 (\partial_iA_{ij}\partial_j\phi_n+ H(x,D\phi_n) -\alpha_n ) \leq 0, 
$$
we also have by coercivity of $H$ and thanks to the bound on $(\alpha_n)$:
$$
\zeta_n(0) -Ct^{\frac1{p'}} \leq \zeta_n(t) \qquad \forall t\in [0,T].
$$
Letting $n\to+\infty$: 
$$
\limsup_n \zeta_n(0) -Ct^{\frac1{p'}} \leq\bar  \zeta(t) \qquad a.e. \ t\in [0,T], 
$$
so that $\ds \limsup_n \zeta_n(0) \leq \inte m_0\bar\phi(0)$. Hence 
$$
\liminf_n \int_0^T\inte F^*(x,\alpha_n)-\inte m_0\phi_n(0) \geq \int_0^T\inte F^*(x,\bar \alpha)-\inte m_0\bar\phi(0)
$$
and $(\bar\phi,\bar\alpha)$ is a minimum. 
\end{proof}

\begin{Remark}\label{rem:CS}{\rm If $r>2$ and $p>1+d/r$, then by \cite{CS} the sequence $(\phi_n)$ built at the beginning of the proof is uniformly H\"{o}lder continuous. Hence so is $\phi$.
}\end{Remark}

\section{Existence and uniqueness of a solution for the MFG system}\label{sec:exuniq}

In this section we show that the MFG system \eqref{MFG} has a unique weak solution and prove the stability of this solution with respect to the data. 

\subsection{Definition of weak solutions}

The variational method described above provides weak solutions for the MFG system. By a weak solution, we mean the following: 

\begin{Definition}\label{def:weaksolMFG} We say that a pair $(\phi,m)\in L^\gamma((0,T)\times\T^d) \times L^q((0,T)\times\T^d)$ is a weak solution to \eqref{MFG} if 
\begin{itemize}
\item[(i)] the following integrability conditions hold:
$$\ds D\phi\in L^r, \; \ds mH^*(\cdot, D_pH(\cdot,D\phi))\in L^1 
\quad{\rm and }\quad m D_pH(\cdot,D\phi))\in L^1 .
$$ 

\item[(ii)] Equation \eqref{MFG}-(i) holds in the following sense: inequality
\be\label{eq:distrib}
\ds \quad -\partial_t \phi-\partial_i (A_{ij}(x)\partial_{j}\phi) + (\partial_i A_{ij})\partial_j \phi+H(x,D\phi)\leq  f(x,m) \quad {\rm in }\; (0,T)\times \T^d, 
\ee
with $\phi(T,\cdot)\leq \phi_T$, holds in the sense of distributions,  

\item[(iii)] Equation \eqref{MFG}-(ii) holds:  
\be\label{eqcontdef}
\ds \quad \partial_t m-\partial_{ij} (A_{ij}(x)m)-{\rm div}(m D_pH(x,D\phi)))= 0\  {\rm in }\; (0,T)\times \T^d, \quad m(0)=m_0
\ee
in the sense of distributions,

\item[(iv)] The following equality holds: 
\be\label{defcondsup} \begin{array}{l}
\ds \int_0^T\inte m(t,x)\left(f(x,m(t,x))+ H^*(x, D_pH(x,D\phi)(t,x)) \right)dxdt\\
\ds\qquad \qquad \qquad \qquad \qquad \qquad \qquad + \inte m(T,x)\phi_T(x)-m_0(x)\phi(0,x)dx=0. 
\end{array}\ee 
\end{itemize}
\end{Definition}

Our main result is the following existence and uniqueness theorem:

\begin{Theorem}\label{theo:mainex} There exists a weak solution $( \phi, m)$ to the MFG system \eqref{MFG}. Moreover this solution is unique in the following sense: if $( \phi, m)$ and $(\phi', m')$ are two solutions, then $m= m'$ a.e. and $ \phi=\phi'$ in $\{ m>0\}$. 

Finally, there exists a solution which is bounded below by a constant depending on $\|\phi_T\|_{C^2}$, on $\|A_{ij}\|_{C^0}$ and on $\|H(\cdot, D\phi_T)\|_\infty$.
\end{Theorem}

\begin{Remark}{\rm Under the assumptions of Remark \ref{rem:CS}, i.e., if $r>2$ and $p>1+d/r$, the $\phi$-component of the solution is  locally H\"{o}lder continuous. 
}\end{Remark}

\subsection{Existence of a weak solution}

The first step towards the proof of Theorem \ref{theo:mainex} consists in showing a one-to-one equivalence between solutions of the MFG system and the two optimizations problems \eqref{Pb:mw2} and \eqref{PB:dual-relaxed}.

\begin{Theorem}\label{theo:main}
Let $(\bar m,\bar w)\in \mathcal K_1$ be a minimizer of \eqref{Pb:mw2} and $(\bar \phi,\bar \alpha)\in \mathcal K$ be a minimizer of \eqref{PB:dual-relaxed}. Then $(\bar \phi, \bar m)$ is a weak solution of the mean field games system \eqref{MFG} and $\bar w= -\bar mD_pH(\cdot,D\bar \phi)$ while $\bar\alpha= f(\cdot,\bar m)$ a.e.. 

Conversely, any weak solution $(\bar \phi,\bar m)$ of \eqref{MFG}  is such that the pair $(\bar m,-\bar mD_pH(\cdot,D\bar\phi))$ is the minimizer of \eqref{Pb:mw2} while $(\bar \phi, f(\cdot,\bar m))$ is a minimizer of \eqref{PB:dual-relaxed}. 
\end{Theorem}

\begin{proof}
Let $(\bar m,\bar w) \in \s{K}_1$ be a minimizer of Problem \eqref{Pb:mw2} and $(\bar \phi,\bar \alpha) \in \s{K}$ be a minimizer of Problem   \eqref{PB:dual-relaxed}.
Due to Lemma \ref{Lem:dualite} and Proposition \ref{prop:valeursegales}, we have
\begin{equation*}
\int_0^T \int_{\bb{T}^d} F^*(x,\bar \alpha) + F(x,\bar m) + \bar mH^*\left(x,-\frac{\bar w}{\bar m}\right) dxdt + \int_{\bb{T}^d} \phi_T \bar m(T) - \bar \phi(0)m_0 dx = 0.
\end{equation*}
We show that $\bar \alpha = f(x,\bar m)$. Indeed, by convexity of $F$,
\begin{equation} \label{eq:F_Fstar_geq}
F^*(x,\bar \alpha(t,x)) + F(x,\bar m(t,x)) - \bar \alpha(t,x)\bar m(t,x) \geq 0, 
\end{equation}
hence
$$
\int_0^T \int_{\bb{T}^d}  \bar \alpha(t,x)\bar m(t,x)+ \bar mH^*\left(x,-\frac{\bar w}{\bar m}\right) dxdt + \int_{\bb{T}^d} \phi_T \bar m(T) -\bar  \phi(0)m_0 dx \leq 0.
$$
Thanks to Lemma \ref{lem:phialphamw}, the above inequality is in fact  an equality,  $\bar w=-\bar mD_pH(\cdot,D\bar \phi)$ a.e. and the equality holds almost everywhere in Equation (\ref{eq:F_Fstar_geq}). Therefore, 
\begin{equation} \label{eq:alpha_equals_f}
\bar \alpha(t,x) = f(x,\bar m(t,x))
\end{equation}
almost everywhere and (\ref{defcondsup}) holds:
\begin{equation*} \label{eq:ibp_equality}
\int_0^T \int_{\bb{T}^d} f \bar m + \bar mH^*\left(x,-\frac{\bar w}{\bar m}\right) dxdt + \int_{\bb{T}^d} \phi_T \bar m(T) - \bar \phi(0)m_0 dx = 0.
\end{equation*}
In particular $\bar mH^*(\cdot, D_pH(\cdot,D\bar \phi))\in L^1$.
 
Moreover, since $(\bar \phi,\bar \alpha) \in \s{K}$ and Equation (\ref{eq:alpha_equals_f}) holds, we have $-\partial_t\bar  \phi -A_{ij}\partial_{ij}\bar \phi+ H(x,D\bar \phi) \leq f(x,\bar m)$ in the sense of distributions and $\bar \phi(T)\leq \phi_T$.

Furthermore, since $(\bar \phi,\bar \alpha) \in \s{K}$  and $\bar w=-\bar mD_pH(\cdot,D\bar \phi)$, we have that $\bar mD_pH(\cdot,D\bar \phi)\in L^1$ and (\ref{eqcontdef}) holds in the sense of distributions.

Therefore $(\bar \phi, \bar m)$ is a solution in the sense of Definition \ref{def:weaksolMFG}.

Suppose now that $(\bar \phi,\bar m)$ is a weak solution of (\ref{MFG}) as in Definition \ref{def:weaksolMFG}. Set $\bar w = -\bar mD_p H(\cdot,D\bar \phi)$ and $\bar \alpha(t,x) = f(x,\bar m(t,x))$. By definition of weak solution $\bar w,\bar \alpha \in L^1$, $\bar m\in L^q$ and $\bar \phi\in L^\gamma$.  Moreover, since $f$ is increasing in $\bar m$ and $\bar m\in L^q$, the growth condition (\ref{HypGrowthF}) implies that $\bar \alpha \in L^p$.
Therefore $(\bar m,\bar w) \in \s{K}_1$ and $(\bar \phi,\bar \alpha) \in \s{K}$.

It remains to show that $(\bar \phi,\bar \alpha)$ minimizes $\s{A}$ and $(\bar m,\bar w)$ minimizes $\s{B}$.

 Let $(\bar \phi',\bar \alpha') \in \s{K}$. By the convexity of $F$ in the second variable, we have
\begin{align*}
\s{A}(\bar \phi',\bar \alpha') &= \int_0^T \int_{\bb{T}^d} F^*(x,\bar \alpha'(t,x))dx dt - \int_{\bb{T}^d} \bar \phi'(0,x)m_0(x)dx\\
&\geq \int_0^T \int_{\bb{T}^d} F^*(x,\bar \alpha(t,x)) + \partial_\alpha F^*(x,\bar \alpha(t,x))(\bar \alpha'(t,x) - \bar \alpha(t,x)) dx dt - \int_{\bb{T}^d}\bar  \phi'(0,x)m_0(x)dx\\
&\geq \int_0^T \int_{\bb{T}^d} F^*(x,\bar \alpha(t,x)) +\bar  m(t,x)(\bar \alpha'(t,x) - \bar \alpha(t,x)) dx dt - \int_{\bb{T}^d} \bar \phi'(0,x)m_0(x)dx,\\
&\geq \s{A}(\bar \phi,\bar \alpha) + \int_0^T \int_{\bb{T}^d}  \bar m(t,x)(\bar \alpha'(t,x) -\bar  \alpha(t,x)) dx dt + \int_{\bb{T}^d}(\bar \phi(0,x)-\bar  \phi'(0,x))m_0(x)dx. 
\end{align*}
Due to  Equation (\ref{defcondsup}) and Lemma \ref{lem:phialphamw} applied to $(\bar \phi',\bar \alpha')$  and $(\bar m,\bar w)$ we have
\begin{align*}
&\int_0^T \int_{\bb{T}^d}  \bar m(t,x)(\bar \alpha'(t,x) - \bar \alpha(t,x)) dx dt + \int_{\bb{T}^d}(\bar \phi(0,x)- \bar \phi'(0,x))m_0(x)dx=\\ &\int_0^T \int_{\bb{T}^d}  \bar m(t,x)\bar \alpha'(t,x) +\bar m(t,x) H^*(x,-\frac{\bar w(t,x)}{\bar m(t,x)}) dx dt + \int_{\bb{T}^d} \phi_T(x)\bar m(T,x)-\bar \phi'(0,x)m_0(x)dx\geq 0.
\end{align*}
Hence,
$$
\s{A}(\bar \phi',\bar \alpha') \geq \s{A}(\bar \phi,\bar \alpha),
$$
and $(\bar \phi,\bar \alpha)$ is a minimizer of $\s{A}$. 

The argument for $(\bar m,\bar w)$ is similar. Let $(\bar m',\bar w')$ minimize $\s{B}$. Then because $F$ is convex in the second variable, we have
\begin{align*}
\s{B}(\bar m',\bar w') &= \int_{\bb{T}^d} \phi_T \bar m'(T) + \iint \bar m'H^*\left(x,-\frac{\bar w'}{\bar m'}\right) + F(x,\bar m')\\
&\geq \int_{\bb{T}^d} \phi_T \bar m'(T) + \iint \bar m'H^*\left(x,-\frac{\bar w'}{\bar m'}\right) + F(x,\bar m) + f(x,\bar m)(\bar m'-\bar m)\\
&= \int_{\bb{T}^d} \phi_T \bar m'(T) + \iint \bar m'H^*\left(x,-\frac{\bar w'}{\bar m'}\right) + F(x,\bar m) + \bar \alpha(\bar m'-\bar m)\\
&=\s{B}(\bar m,\bar w)+ \int_{\bb{T}^d} \phi_T \bar m'(T)-m_0\bar \phi(0) + \iint \bar m'H^*\left(x,-\frac{\bar w'}{\bar m'}\right)  + \bar \alpha \bar m'\\
&\geq \s{B}(\bar m,\bar w).
\end{align*}
Here  we used Equation (\ref{defcondsup}) in the next to last line, and we applied  Lemma \ref{lem:phialphamw} to $(\bar \phi,\bar \alpha)$  and $(\bar m',\bar w')$ in the last line. Therefore $(\bar m,\bar w)$ is a minimizer of $\s{B}$. 

\end{proof}

\subsection{Uniqueness of the weak solution}\label{subsec:uniq}

\begin{proof}[Proof of Theorem \ref{theo:mainex} (uniqueness part)]

Let $(\bar \phi,\bar m)$ be a weak solution to \eqref{MFG}. In view of Theorem \ref{theo:main},  the pair $(\bar m,-\bar mD_pH(\cdot,D\bar \phi))$ is the minimizer of \eqref{Pb:mw2} while $(\bar \phi, f(\cdot,\bar m))$ is a solution of \eqref{PB:dual-relaxed}. In particular,  $\bar m$ is unique because of the uniqueness of the solution of \eqref{Pb:mw2}.

Let now $( \phi_1,\bar m)$ and $( \phi_2,\bar m)$ be two weak solutions of \eqref{MFG}, and set $\bar \alpha = f(\cdot,\bar m)$.
Let $\bar \phi = \phi_1\vee\phi_2$.
Assume for now that $\bar \phi$ is a subsolution of (\ref{ineq:phi}) in the sense of distributions.
Then $(\bar \phi,\bar \alpha) \in \s{K}$, and so because $-\int \bar \phi(0)m_0 \leq - \int \phi_1(0)m_0$ we have that $(\bar \phi,\bar \alpha)$ is also a solution of \eqref{PB:dual-relaxed}.
Indeed, one deduces from Lemma \ref{lem:phialphamw} that for a.e. $t\in [0,T]$, $(\bar \phi,\bar \alpha)$ and $(\phi_1,\bar \alpha)$, are both minimizers of the problem 
$$
\inf_{(\phi, \alpha)\in {\mathcal K}} \int_t^T\inte F^*(x,\alpha)-\inte m(t)\phi(t).
$$
In particular, $\ds \inte\bar  m(t)\bar \phi(t) = \inte \bar m(t)\phi_1(t)$. As $\phi_1 \leq\bar  \phi$, this implies that $\phi_1 =\bar \phi$ a.e. in $\{\bar m>0\}$. 
The same argument, replacing $\phi_1$ with $\phi_2$, shows that $\phi_2 =\bar \phi$ a.e. in $\{\bar m>0\}$, and uniqueness is proved.

The main work to be shown is that $\bar \phi = \phi_1\vee \phi_2$ is indeed a subsolution of (\ref{ineq:phi}) in the sense of distributions, i.e.
\begin{equation}
\label{eq:max-subsolution}
-\int_{\T^d}\zeta(T)\phi_T + \int_0^T\int_{\T^d}\bar  \phi\partial_t\zeta+ \lg D\zeta, AD\bar \phi\rg + \zeta(\partial_iA_{ij}\partial_j\bar \phi +H(x,D\bar \phi)) \leq \int_0^T\int_{\T^d} \bar \alpha \zeta
\end{equation}
for any nonnegative smooth map $\zeta$ with support in $(0,T]\times \T^d$.

Let $\epsilon > 0$.
Introduce the following translation and extension of $(\phi_k,\bar \alpha), k=1,2$:
\begin{equation}
\tilde{\phi}_k(t,x) = \left\{
\begin{array}{ll}
\phi_k(t+2\epsilon,x) & \text{if}~t \in [-2\epsilon,T-2\epsilon)\\
\phi_T(x) & \text{if}~t\in [T-2\epsilon,T+2\epsilon]
\end{array}\right.
\end{equation}
 and
\begin{equation}
\tilde{\alpha}(t,x) = \left\{
\begin{array}{ll}
\bar \alpha(t+2\epsilon,x) & \text{if}~t \in [-2\epsilon,T-2\epsilon)\\
\lambda & \text{if}~t\in [T-2\epsilon,T+2\epsilon]
\end{array}\right.
\end{equation}
where $\lambda = \max_x H(x,D\phi_T(x)) + A_{ij}(x) \partial_{ij}  \phi_T(x)$.
Then we have that
\begin{equation}
-\partial_t \tilde{\phi}_k - A_{ij}\partial_{ij}\tilde{\phi}_k + H(x,D\tilde{\phi}_k) \leq \tilde{\alpha}
\end{equation}
in the sense of distributions on $(-2\epsilon,T+2\epsilon) \times \bb{T}^d$.

For now we will fix a smooth vector field $\psi$ on $[0,T] \times \bb{T}^d$.
Notice that
\begin{equation}
-\partial_t \tilde{\phi}_k - A_{ij}\partial_{ij}\tilde{\phi}_k + \psi \cdot D\tilde{\phi}_k \leq \tilde{\alpha} + H^*(x,\psi)
\end{equation}
in the sense of distributions on $(-2\epsilon,T+2\epsilon) \times \bb{T}^d$.

Let $\xi^1$ be a smooth convolution kernel in $\bb{R}^{d+1}$ with support in the unit ball, with $\xi^1 \geq 0$ and $\int \xi^1 = 1$.
Then define the standard mollifier sequence $\xi^\epsilon(t,x) = \epsilon^{-d-1}\xi^1((t,x)/\epsilon)$.
Set $\phi_k^\epsilon = \xi^\epsilon \star \tilde{\phi}_k$ and $\alpha^\epsilon = \xi^\epsilon \star \tilde \alpha$.
By taking the convolution we have, in a pointwise sense,
\begin{equation} \label{eq:subsol+commutators}
-\partial_t \phi_k^\epsilon - A_{ij}\partial_{ij}\phi_k^\epsilon + \psi \cdot D\phi_k^\epsilon \leq \alpha^\epsilon + \xi^\epsilon \star H^*(\cdot,\psi) + R_\epsilon^k - S_\epsilon^k
\end{equation}
on $[0,T] \times \bb{T}^d$, where
\begin{equation}
R_\epsilon^k := [\xi^\epsilon,A_{ij}\partial_{j}](\partial_i \tilde{\phi}_k), ~~ S_\epsilon^k := [\xi^\epsilon,\psi ](D\tilde{\phi}_k).
\end{equation}
Here we use the same commutator notation as in (\ref{eq:commutators}).
Invoking \cite[Lemma II.1]{DL89}, we have that $R_\epsilon^k$ and $S_\epsilon^k$, $k=1,2$ are smooth functions which converge to zero in $L^{r}$, since $A_{ij} \in W^{1,\infty}$ is given and $\psi$ may also be chosen in $W^{1,\infty}$.

Define $R_\epsilon := \max\{R_\epsilon^1 - S_\epsilon^1,R_\epsilon^2 - S_\epsilon^2\}$.
This, too, converges to zero in $L^r$.
Moreover, for $k=1,2$
\begin{equation} \label{eq:subsol+error}
-\partial_t \phi_k^\epsilon - A_{ij}\partial_{ij}\phi_k^\epsilon + \psi \cdot D\phi_k^\epsilon \leq \alpha^\epsilon + \xi^\epsilon \star H^*(\cdot,\psi) + R_\epsilon
\end{equation}
holds in a pointwise sense, hence also in a viscosity sense.
By standard results, (\ref{eq:subsol+error}) holds also for $\phi^\epsilon := \phi^\epsilon_1\vee \phi^\epsilon_2$ in a viscosity sense.
The result of \cite{Ish95} implies that it also holds in the sense of distributions, that is, for any smooth map $\zeta$ with support in $(0,T] \times \bb{T}^d$ we have
\begin{multline} \label{eq:subsolindist+error}
-\int_{\T^d}\zeta(T)\phi^\epsilon(T) + \int_0^T\int_{\T^d} \phi^\epsilon\partial_t\zeta+ \lg D\zeta, AD\phi^\epsilon\rg + \zeta(\partial_iA_{ij}\partial_j\phi^\epsilon +D\phi^\epsilon \cdot \psi) \\
\leq \int_0^T\int_{\T^d} \zeta(\alpha^\epsilon  + \xi^\epsilon \star H^*(\cdot,\psi) + R_\epsilon).
\end{multline}
By construction, $\phi^\epsilon(T) = \phi_T$ for all $\epsilon > 0$.
Observe that $\phi^\epsilon \to\bar \phi$ in $L^\gamma$ and $D\phi^\epsilon \to D\bar \phi$ in $L^r$, as these sequences are only slight adaptations of classical convolutions of $\bar \phi$ and $D\bar \phi$.
Finally, note that $\alpha^\epsilon \to\bar  \alpha$ in $L^p$, while $\xi^\epsilon \star H^*(\cdot,\psi) \to H^*(\cdot,\psi)$ uniformly.
Letting $\epsilon \to 0+$, we are left with
\begin{equation} \label{eq:0-estimate}
-\int_{\T^d}\zeta(T)\phi_T + \int_0^T\int_{\T^d} \bar \phi\partial_t\zeta+ \lg D\zeta, AD\bar \phi\rg + \zeta(\partial_iA_{ij}\partial_j\bar \phi +\psi \cdot D\bar \phi)
 \leq \int_0^T\int_{\T^d}  \zeta(\bar \alpha +  H^*(\cdot,\psi)).
\end{equation}
Now since $\psi$ is an arbitrary smooth vector field, we may take a sequence that approximates $\partial_p H(x,D\bar \phi)$ in $L^{r'}$.
By the convexity of $H(x,\cdot)$ this yields (\ref{eq:max-subsolution}), as desired.
\end{proof}

\subsection{Stability}

We now consider the stability of solutions with respect to  the data $A$, $H$ and $f$ and the data $m_0$ and $\phi_T$. More precisely, assume that $(A^n)$, $(H^n)$, $(f^n)$ $m_0^n$ and $\phi^n_T$ satisfy conditions (H1)$\dots$(H4) uniformly with respect to $n$ and converge to $A$, $H$, $f$, $m_0$ and $\phi_T$ locally uniformly.

\begin{Theorem}\label{thm:stab} Let $(\phi^n,m^n)$ be a weak solution of \eqref{MFG} associated with $A^n$, $H^n$,  $f^n$ and with the initial and terminal conditions $m_0^n$ and $\phi_T^n$. Assume also that the sequence $\phi^n$ is uniformly bounded below. Then $(m^n)$ converges strongly to $m$ in $L^q$ while $\phi^n$ converges weakly and up to a subsequence to a map $\bar \phi$ in $L^\gamma$, where the pair $(\bar \phi,\bar m)$ is a weak solution to \eqref{MFG}. 
\end{Theorem}

Note that the existence of a solution $(\phi^n,m^n)$, such that $\phi^n$ is bounded by below, is ensured by Theorem \ref{theo:mainex}.
 
The result is a simple consequence of Theorem \ref{theo:main} and of the $\Gamma-$convergence of the corresponding variational problems.  

\begin{proof} Let us set $w^n=-m^nD_pH_n(\cdot,D\phi^n))$ and $\alpha^n= f(\cdot,m^n)$. According to the second part of Theorem \ref{theo:main}, the pair $(m^n,w^n)$ is a minimizer of problem  \eqref{Pb:mw2} associated with $A^n$, $H^n$, $f^n$, $m_0^n$ and $\phi_T^n$, while the pair $(\phi^n,\alpha^n)$ is a minimizer of problem \eqref{PB:dual-relaxed} associated with the same data. Using the estimates established for the proof of Proposition \ref{prop:valeursegales}, we have 
\be\label{boundmnwn}
\|m^n\|_{L^q}+ \|w^n\|_{L^{\frac{r'q}{r'+q-1}}} \leq C. 
\ee
By lower semi-continuity of the functional $\mathcal B$, $(m^n,w^n)$ converge weakly up to a subsequence to to the minimum $(\bar m, \bar w)$ of the problem \eqref{Pb:mw2} associated with $A$, $H$, $f$, $m_0$ and $\phi_T$. The limit problem being strictly convex, the convergence actually holds 
strongly in $L^q\times L^{\frac{r'q}{r'+q-1}}$.  

Then the growth condition \eqref{Hypf} on $f$ implies that the sequence $(\alpha^n=f^n(\cdot,m^n))$  converges in $L^p$ to $\bar \alpha:= f(\cdot,\bar m)$. 
As $(\phi^n)$ is uniformly bounded below, Theorem \ref{thm:estiHJ} implies that 
$$
\|\phi^n\|_{L^\infty((0,T), L^{\eta}(\T^d))}+ \|\phi^n\|_{L^\gamma((0,T)\times \T^d)}\leq C.
$$
The end of the proof follows closely the argument in Proposition \ref{Prop:existence}:  $(D\phi^n)$ is bounded in $L^r$, so that, up to a subsequence, $(\phi^n)$ converges weakly to some $\bar \phi$ in $L^\gamma$ while $D\phi^n$ converges weakly to $D\phi$ in $L^r$ where $(\bar \phi,\bar \alpha)$ belongs to  ${\mathcal K}$. Moreover $(\bar \phi,\bar \alpha)$ is a minimizer of the relaxed problem \eqref{PB:dual-relaxed}. Theorem \ref{theo:main} then states that the pair $(\bar \phi,\bar m)$ is a solution to the MFG problem \eqref{MFG}. 

\end{proof}

%


\begin{thebibliography}{abc99xyz}
\bibitem{AGS}  Ambrosio, L., Gigli, N., Savar\`e, G. Gradient flows in metric spaces and in the space of probability measures. Lectures in Mathematics ETH Z\"urich. Birkh\"auser Verlag, Basel, 2008.

%

\bibitem{bb} Benamou J.D.,  Brenier Y., {\it A computational fluid mechanics solution to the Monge-Kantorovich mass transfer problem}, {Numer. Math.}, 84, 375--393, (2000). 




\bibitem{CaDe2} R. Carmona and F. Delarue. {\it Probabilist analysis of Mean-Field Games.} Preprint,  2013.  




\bibitem{c1} Cardaliaguet P.
{\it Weak solutions for first order mean field games with local coupling.} Preprint 

\bibitem{cg} Cardaliaguet P. and Graber P.~J.
{\it Mean field games systems of first order.} Preprint 



\bibitem{CCN} Cardaliaguet P., Carlier G.,  Nazaret B., {\it Geodesics for a class of distances in the space of probability measures.}  Calculus of Variations and Partial Differential Equations,  (2012) 1-26.

\bibitem{CS} Cardaliaguet P., Silvestre L. {\it H\"older continuity to Hamilton-Jacobi equations with superquadratic growth in the gradient and unbounded right-hand side.} Communications in Partial Differential Equations,  vol. 37  (2012), no 9, p. 1668-1688.

\bibitem{CLLP} Cardaliaguet P., Lasry J.-M., Lions P.-L., Porretta A. {\it Long time average of mean field games}. 
 Networks and Heterogeneous Media 
7 (2012), no. 2,  279-301.

\bibitem{DiB} DiBenedetto E., Degenerate parabolic equations, Springer-Verlag, 1993.

\bibitem{DL89}
DiPerna R. and Lions P.-L., \emph{Ordinary differential equations,
  transport theory and {S}obolev spaces}, Inventiones mathematicae, 98
  (1989), no.~3, 511--547.

\bibitem{ET} Ekeland, I., and T\`emam, R. Convex analysis and variational problems, english ed., vol. 28 of Classics
in Applied Mathematics. Society for Industrial and Applied Mathematics (SIAM), Philadelphia,
PA, 1999. Translated from the French.


\bibitem{GPSM1} Gomes, D. A., Pimentel, E., S\'{a}nchez-Morgado, H. {\it Time dependent mean-field games in the subquadratic case.}
 arXiv preprint arXiv:1311.6684, 2013.

\bibitem{GPSM2} Gomes, D. A., Pimentel, E., S\'{a}nchez-Morgado, H. {\it Time dependent mean-field games in the superquadratic case.}
 arXiv preprint arXiv:1311.6684, 2013.

\bibitem{Gra14}
Graber P.~J., \emph{Optimal control of first-order {Hamilton--Jacobi}
  equations with linearly bounded {Hamiltonian}}, Applied Mathematics \&
  Optimization (2014), 1--40.

\bibitem{HCMieeeAC06} Huang, M., Malham\'e, R.P., Caines, P.E.  (2006). 
{\it Large population stochastic dynamic games: closed-loop McKean-Vlasov systems and the Nash certainty equivalence
principle.} Communication in information and systems. Vol. 6, No. 3, pp. 221-252. 

\bibitem{Ish95}
Ishii H., \emph{On the equivalence of two notions of weak solutions,
  viscosity solutions and distribution solutions}, Funkcial. Ekvac \textbf{38}
  (1995), no.~1, 101--120.

\bibitem{LL06cr1} Lasry, J.-M., Lions, P.-L. {\it Jeux \`a champ moyen. I. Le cas stationnaire.}
C. R. Math. Acad. Sci. Paris  343  (2006),  no. 9, 619-625.

\bibitem{LL06cr2} Lasry, J.-M., Lions, P.-L. {\it Jeux \`a champ moyen. II. Horizon fini et contr\^ole optimal.}
C. R. Math. Acad. Sci. Paris  343  (2006),  no. 10, 679-684.

\bibitem{LL07mf} Lasry, J.-M., Lions, P.-L. {\it Mean field games.}  Jpn. J. Math.  2  (2007),  no. 1, 229--260.

\bibitem{BL08}
Le Bris C. and Lions P.-L., \emph{Existence and uniqueness of solutions to
  {Fokker--Planck} type equations with irregular coefficients}, Communications
  in Partial Differential Equations \textbf{33} (2008), no.~7, 1272--1317.

\bibitem{LLperso} Lions, P.L. In Cours au Coll\`{e}ge de France. 
www.college-de-france.fr.

\bibitem{Po} Porretta A. \emph{Weak solutions to Fokker-Planck equations and mean field games.} Preprint (2013). 


\bibitem{St}   Stampacchia G.,  {\it Le probl\`eme de Dirichlet pour les \'equations elliptiques du seconde ordre \`a coefficients 
discontinus} Ann. Inst. Fourier (Grenoble) {\bf 15} (1965), 189--258. 
\end{thebibliography}
\end{document}